\documentclass[reqno,11pt]{amsart}
\usepackage{a4wide}
\usepackage{graphicx}
\usepackage{amsfonts,amsmath,amssymb,amsthm}
\usepackage{mathtools}
\usepackage{pgf,tikz}
\usetikzlibrary{arrows.meta}
\usetikzlibrary{math}
\pgfdeclarelayer{bg}    % declare background layer
\pgfsetlayers{bg,main} 
\usepackage{hyperref}
\usepackage[normalem]{ulem}
\usepackage{color}
\usepackage{ifpdf}
\ifpdf
\usepackage[latin1]{inputenc}
\numberwithin{equation}{section}

\newtheorem{theorem}{Theorem}[section]
\newtheorem{lemma}[theorem]{Lemma}
\newtheorem{proposition}[theorem]{Proposition}
\newtheorem{corollary}[theorem]{Corollary}
\newtheorem{definition}[theorem]{Definition}
\newtheorem{remark}[theorem]{Remark}
\newtheorem*{Conjecture}{Conjecture}
\renewcommand\tilde{\widetilde}

\newcommand\calI{\mathcal{I}}
\newcommand\calC{\mathcal{C}}

\def\R{\mathbb{R}}

\def\pa{\partial}

\newcommand\calA{\mathcal{A}}

\newcommand\m{\mu}
\def\X{\mathbb{X}}

\def\E{\mathcal{E}}

\def\Z{\mathbb{Z}}
\def\N{\mathbb{N}}
\def\M{\mathcal{M}}
\def\H{\mathcal{H}}

\def\F{\mathcal{F}}

\def\LM#1{\hbox{\vrule width.2pt \vbox to#1pt{\vfill \hrule width#1pt
			height.2pt}}}
\def\LL{{\mathchoice {\>\LM7\>}{\>\LM7\>}{\,\LM5\,}{\,\LM{3.35}\,}}}
\def\restr{{\LL}}
\renewcommand{\phi}{\varphi}

\def\Div{\nabla \cdot\,}

\def\1{\mathbf{1}}

\def\XXint#1#2#3{{\setbox0=\hbox{$#1{#2#3}{\int}$ }
		\vcenter{\hbox{$#2#3$ }}\kern-.57\wd0}}

\def\eps{\varepsilon}
\renewcommand{\subset}{\subseteq}

\def\lt{\left}
\def\rt{\right}
\def\les{\lesssim}
\def\ges{\gtrsim}

\def\what{\widehat}

\def\dim{{\rm dim}}
\def\dimlX{\underline{\dim}_{\X}\,}
\newcommand{\nper}[1]{|#1|_{\rm per}}
\def\dimX{\dim_{\X}\,}
\def\Ecin{E_{\rm cin}}
\def\dimuX{\overline{\dim}_{\X}\,}
\def\per{\rm per}
\newcommand{\dw}{\downarrow}
\newcommand{\up}{\uparrow}
\def\spt{{\rm Spt }\,}
\def\Eund{\E^{\rm 1d}}

\def\cL{\mathcal{L}}
\def\dotEcin{\dot{E}_{\rm cin}}

\title[From energy bounds to dimensional estimates]{From energy bounds to dimensional estimates in a branched transport model for type-I superconductors}

\author[G. De Philippis]{Guido De Philippis}
\address{Courant Institute of Mathematical Sciences, NYU ,251 Mercer St., New York, NY 10012, USA.}
\email{guido@cims.nyu.edu}

\author{Michael Goldman}
\address{CMAP, CNRS, \'Ecole polytechnique, Institut Polytechnique de Paris, 91120 Palaiseau,
France}
\email{michael.goldman@cnrs.fr}

\author{Berardo Ruffini}
\address{Department of Mathematics, University of Bologna, 40126 Bologna, Italy} 
\email{berardo.ruffini@unibo.it}

\begin{document}
\begin{abstract}
We consider a branched transport type problem which describes the magnetic flux through type-I superconductors in a regime of very weak applied fields. At the boundary of the sample, deviation of the magnetization from being uniform is penalized through a negative Sobolev norm.
% A central object of study is the trace of the magnetization at the boundary of the sample.
% In this model, non-uniform magnetization is penalized through a negative Sobolev norm.
It was conjectured by S. Conti, F. Otto and S. Serfaty that as a result, the trace of the magnetization on the boundary should be a measure of Hausdorff dimension $8/5$. We prove that this conjecture is equivalent to the proof of local energy bounds with an optimal exponent. We then obtain local bounds which are however not optimal. These yield improved lower bounds on the dimension of the irrigated measure but unfortunately does not improve on the trivial upper bound. In order to illustrate the dependence of this dimension on the choice of penalization,  we consider in the last part of the paper a toy model where the boundary energy is given by a Wasserstein distance to Lebesgue. In this case minimizers are finite graphs and thus the trace is atomic.
	\end{abstract} 
	\maketitle
	
	%\tableofcontents

\section{Introduction}
It has been observed experimentally \cite{Proz,Proal,ProzHo} that when subject to an external magnetic field, complex patterns appear at the boundary of type-I superconductors. It can be argued,  \cite{Proal,ChokKoOt}, that in the regime of vanishing external field these reflect the structure of low energy states. Based on the Ginzburg-Landau theory, it is known since the work of Landau \cite{Landau} that inside the sample these should be branching patterns. This has first been rigorously justified at the level of scaling laws in \cite{ChokKoOt,ChokConKoOt,CoOtSer}. Later on, in the so-called uniform branching regime (see below), a branched transport type problem has been derived from the full Ginzburg-Landau energy in \cite{CGOS}. This  gives a full justification of the emergence of branching patterns in this part of the phase diagram. In a simplified two-dimensional setting, the minimizers of the reduced problem (for thick enough samples) have been completely characterized in \cite{G}. They are self-similar, branching trees.\\
In this paper we consider the cross-over regime between the uniform and non-uniform branching regimes. Starting from the corresponding reduced branched transport model (see \eqref{def:enerintro} below) we aim at  a qualitative description of the irrigated measure. In particular we are interested in its Hausdorff dimension. Based on heuristics that we recall in Section \ref{sec:heur}, it has been conjectured in \cite{privateCOS} that this measure should be of dimension $8/5$. Our main result, Theorem \ref{thm:main1}, reduces the proof of this conjecture to the proof of a local scaling law for the energy inside the sample. This is complemented in Theorem \ref{thm:main2} by a (a priori non sharp) local scaling law which albeit not proving the  conjecture,  reduces the range of possible dimensions to the interval $[3/2,2]$.  By itself, these results can be seen as a justification of the belief that the complexity of the patterns observed at the surface of the sample are related to the patterns living inside the sample.

\subsection{The model}
As in \cite{ChokKoOt,ChokConKoOt,CoOtSer,CGOS}, we consider the simplest possible geometric configuration of a plate $Q_T=(-T,T)\times Q$ with thickness $2T$ and periodic lateral boundary conditions i.e. $Q=(\R/\Z)^2$ subject to a perpendicular external magnetic field. Having \cite{CGOS} in mind we expect that the normal phase concentrate on branched structures  in the regime of vanishing external magnetic field. We thus represent it by a measure $\mu=dt \otimes \mu_t$ such that for a.e. $t\in (-T,T)$,
\[
 \mu_t=\sum_i \phi_i \delta_{X_i}
\]
where $X_i\in Q$ and $\phi_i\ge 0$ are such that $\sum_i \phi_i=1$. A bit formally (see Section \ref{sec:internal} for a precise definition) we set for $-T\le a<b\le T$
\begin{equation}\label{defI}
 I(\mu,(a,b))=\int_{a}^b \sum_i\lt[\phi_i^{\frac{1}{2}}+ \phi_i |\dot{X}_i|^2\rt] dt,
\end{equation}
where $\dot{X}_i$ denotes the derivative of the map $t\mapsto X_i(t)$. When $b=-a=T$, we simply write
\[
 I(\mu)=I(\mu,(-T,T)).
\]
As explained in \cite[Section 5]{CGOS}, this is a variant of the branched transport functionals which have received a lot of attention in the last couple of decades, see e.g.  \cite{BeCaMo,MSM,BrBuS,BraWirth,BraSol,Colo,pegfract,pegLag,SanLan}. Using the short-hand notation (with $H^{-1/2}=H^{-1/2}(Q)$)
\[
 \|\mu_{\pm T}-1\|_{H^{-\frac{1}{2}}}^2=\|\mu_{- T}-1\|_{H^{-\frac{1}{2}}}^2+\|\mu_{ T}-1\|_{H^{-\frac{1}{2}}}^2,
\]
 we then consider for $\lambda>0$ the energy
\begin{equation}\label{def:enerintro}
 \E_{\lambda,T}(\mu)=I(\mu) + \lambda \|\mu_{\pm T}-1\|_{H^{-\frac{1}{2}}}^2.
 \end{equation}
By a simple symmetrization argument we will always consider symmetric minimizers of $\E_{\lambda,T}$, see \cite{CGOS}.\\

In the case $\lambda=\infty$, that is under the constraint that $\mu_{\pm T}=1$ (hence the name uniform branching regime), this functional has been rigorously derived in \cite{CGOS} from the full Ginzburg-Landau energy.
In the notation of \cite{CGOS}, if $\beta$ represents the strength of the external magnetic field and $\alpha^{-1}$ the coherence length, the uniform branching regime corresponds to $\alpha^{-2/7}\ll \beta\ll1$. Formal computations suggest that  \eqref{def:enerintro} should be the appropriate reduced energy in the crossover regime $\alpha^{-2/7}\sim \beta\ll1$. We leave the rigorous derivation of the model for future work and  take instead \eqref{def:enerintro} as our starting point. Based on the constructions for the upper bounds in \cite{ChokKoOt,CoOtSer}, the following conjecture was made in \cite{privateCOS}:
\begin{Conjecture}
 If $\mu$ is a minimizer of \eqref{def:enerintro}, then $\mu_{\pm T}$ is of dimension $8/5$.
\end{Conjecture}

\subsection{Results}

Setting for $\eps\in (0,T)$
\begin{equation}\label{defImueps}
 I(\mu,\eps)=I(\mu, (T-\eps,T))
\end{equation}
to be the energy concentrated near the boundary and referring to Section \ref{sec:Sob} for the precise definition of lower and upper Hausdorff dimensions of a measure, our main result reads as follows:

 \begin{theorem}\label{thm:main1}
 	Let $\mu$ be a symmetric minimizer of \eqref{def:enerintro} i.e. with $\mu_{-t}=\mu_t$ for $t\in[0,T]$. If $\beta\in(0,1)$ is such that
 	\begin{equation}\label{hypIepsintro}
 	\limsup_{\eps\to0^+} \frac{I(\mu,\eps)}{\eps^\beta} <\infty,
 	\end{equation}
 	then
 	\begin{equation}\label{eq:lowupbounds}
 	f(\beta)=\frac{1+3\beta}{1+\beta}\le \underline{\dim}\, \mu_T\le  \overline{\dim}\, \mu_T\le \frac{4(1-\beta)}{1+\beta}=g(\beta).
 	\end{equation}
 \end{theorem}
 Before commenting on its  proof, let us discuss the statement and its implications. First, \eqref{eq:lowupbounds} shows that local energy bounds in the form of \eqref{hypIepsintro} directly  translate into lower and upper bounds for the dimension of $\mu_T$. Regarding the lower bound, since  the function $f$ is increasing with $f(0)=1$, establishing \eqref{hypIepsintro} for any $\beta>0$ improves on the bound $1\le \underline{\dim}\, \mu_T $ given for free by $\E_{\lambda,T}(\mu)<\infty$ (see Lemma \ref{lemdimper}). As for the upper bound, since  the function $g$ is decreasing with $g(1/3)=2$, in order to improve on the trivial bound $\overline{\dim}\, \mu_T\le 2$, it is necessary to obtain \eqref{hypIepsintro} for some  $\beta>1/3$. More importantly, since  $f(3/7)=g(3/7)=8/5$, see Figure \ref{fig:fg}, we have the following corollary.

\begin{figure}[h]   \label{fig:fg}
\centering
\begin{tikzpicture}

\begin{scope}[scale=12]

% Tracé de axes
\pgfmathsetmacro{\xa}{0}
\pgfmathsetmacro{\xb}{0.5}

\pgfmathsetmacro{\ya}{0.15}
\pgfmathsetmacro{\yb}{0.5}
\draw[very thin,->] (\xa,\ya) -- (\xb,\ya) node[below]{$\beta$};
\draw[very thin,->] (0,\ya) -- (0,\yb);

% Rectangle plat
\draw[smooth]  plot[domain={0:3/7}]({\x}, {0.2*(1+3*\x)/(1+\x)}) node[below]{$f(\beta)$}; 

\draw[smooth]  plot[domain={4/13:3/7}]({\x}, {0.2*4*(1-\x)/(1+\x)}); 
\draw[dashed]  plot[domain={3/13:4/13}]({\x}, {0.2*4*(1-\x)/(1+\x)}) node[left]{$g(\beta)$}; 
%\draw[thin]  (3/13,0.2*18/7)  node[above]{$g(\beta)$};
\draw[dashed]  plot[domain={0:3/7}]({\x}, {0.2});
\draw[thin] (0,0.2) node[left]{$1$};
 
\draw[dashed]  plot[domain={0:3/7}]({\x}, {0.4});
\draw[thin] (0,0.4) node[left]{$2$};
\draw[dashed] plot[domain={\ya:\yb}]({1/3},{\x});
\draw[thin] (1/3,\ya) node[below]{$\frac{1}{3}$};

\draw[thin] (0,0.2*8/5) node{$-$} node[left]{$\frac{8}{5}$};
\draw[dashed]  plot[domain={0:3/7}]({\x}, {0.2*8/5});
\end{scope}
\end{tikzpicture}
\caption{ The functions $f$ and $g$.} 
\end{figure}
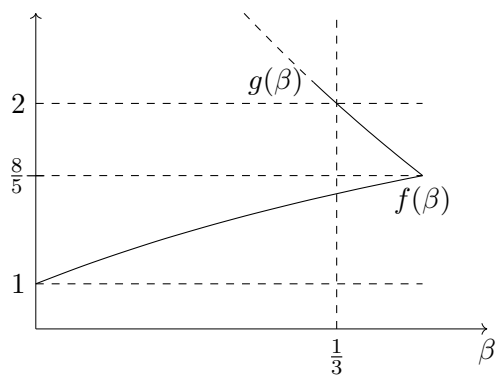

\begin{corollary}
If \eqref{hypIepsintro} holds with $\beta={3/7}$, then  $\dim\, \mu_T= 8/5$.
\end{corollary}
Let us point out  that the scaling $I(\mu,\eps)\les_{\lambda,T} \eps^{3/7}$ is exactly the local version of the global scaling law from Proposition \ref{prop:scaling}.\\
We now outline the proof of Theorem \ref{thm:main1}. The lower bound estimate is given in Corollary \ref{cor:lower}. This may be seen as a regularity result for $\mu_T$. Indeed, we prove in Theorem \ref{thm:lower} that $\mu_T\in H^{-\gamma}$  for an explicit range of parameters $\gamma<1/2$ (depending on $\beta$). We then rely on a variant of Frostman Lemma, see Lemma \ref{lemdimper}, to relate it to an estimate on $\underline{\dim} \, \mu_T$. To prove that $\mu_T\in H^{-\gamma}$ we construct a competitor  by replacing $\mu$  in $(T-\eps,T)$ by two slightly translated copies of $\mu$. Since $I(\mu,\eps)$ is small by hypothesis, this comes with a small cost. By minimality, we can then estimate by how much the $H^{-1/2}$ norm of $\mu_T$ is modified under this procedure. Notice that our  variations  modify $\mu$ up to the boundary so that we use minimality for the full energy $\E_{\lambda,T}$ and not only for the interior energy $I$. The crucial observation is that this variation can be seen as a sort of mollification of $\mu_T$. We can thus rely on a semi-goup characterization of negative Sobolev norms to conclude, see Proposition \ref{prop:H-ssemigroup} and Remark \ref{rem:semigroup}. The upper bound estimate is given in Theorem \ref{thm:upper}. As opposed to the lower bound, this estimate does not rely at all on the minimality of $\mu$. It quantifies instead the intuitive statement that the more diffused is a measure, the more expensive it is to irrigate it. The argument is a generalization of the interpolation type estimates from \cite[Proposition 5.3]{CGOS} (which are used to prove  the ansatz free lower bounds for the global scaling laws) in the case of a measure which is not the Lebesgue measure. Interestingly, this reduces the problem to a non-standard quantization problem, see Remark \ref{rem:quant}.\\
Let us point out that Theorem \ref{thm:main1} is  somewhat similar in spirit to \cite{pegfract} where the ``unit ball'' for the usual branched transport problem is considered. In \cite{pegfract}, a regularity result for the landscape function introduced in \cite{SanLan} is proven which is then translated into an upper bound on the dimension of the boundary of this unit ball. There is  however no complementary lower bound.  While the proofs and functionals are very different, these are, to the best of our knowledge, the only results about the fractal nature of an irrigated measure. Other results such as \cite{BraSol,G,MaMa} are more concerned with the fractal nature of the branching measure itself.\\

Our second main result is that \eqref{hypIepsintro} holds for $\beta=1/3$. In light of the discussion above, this allows to improve on the lower bound on the dimension but unfortunately not on the upper bound. We actually prove a slightly stronger statement.

\begin{theorem}\label{thm:main2}
	Let $\bar \mu$ be any probability measure on $Q$ and let $\mu$ be a symmetric minimizer of
	\[
	\min\lt\{ I(\mu) \ : \mu_{\pm T}=\bar \mu\rt\}.
	\]
	Then, for every $\eps\in (0,T)$,
	\[
	I(\mu,\eps)\les \max\lt(\eps,\eps^{\frac{1}{3}}\rt) + I(\mu) \frac{\eps}{T}.
	\]
	In particular, if $\mu$ is a minimizer of $\E_{\lambda,T}$,
	\[
	\underline{\dim}\, \mu_T\ge \frac{3}{2}.
	\]
\end{theorem}
Since  by Theorem \ref{thm:upper},  $I(\mu,\eps)\ges \eps^{1/3}$ when $\mu_T$ is the Lebesgue measure, this shows that the worst possible scaling is attained when the irrigated measure is uniform. The general strategy of the proof is similar to 
\cite[Theorem 2.6]{KM94}. It consists, given the traces $\mu_{T-\eps}$ and $\mu_T$ in the construction of a competitor in $(T-\eps,T)$ which has kinetic energy almost equal to the relaxed energy (here $W_{\per}^2(\mu_{T-\eps},\mu_T)/\eps$), where $W_{\per}^2$ is the squared $2-$Wasserstein distance on the torus) and a well controlled perimeter. As in \cite{KM94} a natural idea is to first discretize the solution of the relaxed problem (here McCann's interpolant between $\mu_{T-\eps}$ and $\mu_T$) at some well-chosen times $t_k$ with $t_k\to T$ as $k\to \infty$ and then make a construction in $(t_k,t_{k+1})$.  A major difference between \cite{KM94} and our context is that in the case of a transport type energy we are a priori  very constrained regarding the geometrical structure of competitors. In particular, the graph of a typical transport plan between two measures supported on $N$ points has $N^2$ edges which is much too large. The key observation is that among all optimal transport plans (which in particular have the smallest possible kinetic energy among all measure interpolating the discretized measure) there is at least one which is \textit{sparse} in the sense that it contains of the order of $N$ edges, see \cite[Proposition 3.4]{PeyCut}.
\begin{remark}
 As in \cite{KM94} we use the equipartition of the energy , see Proposition \ref{prop:equipar} in the proof. We believe that as in \cite{Conti00} it should be possible to avoid this argument and get a sharper bound i.e. replacing $C/\eta$ by $C$ in \eqref{eq:toprovelocalscaling}. This would require much more work but would be a necessary first step in order to obtain energy bounds which are localized both in time and space as in \cite{Conti00,AlChokOt,Viehmanndiss,CDKZ}. 
\end{remark}

\subsection{A toy model}
The proof of Theorem \ref{thm:main1} can be extended to any space dimension and to any boundary term of the form $H^{-s}$ with $s>0$. The expected dimension of $\mu_T$ however depends on the specific choice of $s$.
In order to highlight
%the fact that the structure of $\mu_T$ strongly depends on the choice of the penalization term,
 this phenomenon, we consider in Section \ref{sec:wasserstein} a toy model where we replace the $H^{-1/2}$ norm by a Wasserstein distance to the Lebesgue measure. This acts morally speaking like a $H^{-1}$ penalization. Moreover, we work in a simplified two-dimensional setting reminiscent of \cite{G}. For $\mu=dt \otimes \mu_t$ with $\mu_t=\sum_{i=1}^{N} \phi_i \delta_{X_i}$ where as before $\sum_{i=1}^N \phi_i=1$ but $X_i\in Q=\R/\Z$, we set
\[
 \Eund_{\lambda,T}(\mu)=\int_{-T}^T \sum_{i=1}^N \lt[1+ \phi_i |\dot{X}_i|^2\rt] dt + \lambda W_{\per}^2(\mu_{\pm T}, 1).
\]
As in \cite{G}, we replaced in the energy the term $\phi_i^{1/2}$ by $\phi_i^0$. This is motivated by the fact that this term should be seen as a perimeter, see Remark \ref{rem:alpha}. Our main result is that for any value of $\lambda$ and $T$, the minimizers of $\Eund_{\lambda,T}$ have a finite number of branching points.

\begin{theorem}\label{thm:was1}
	Let  $\lambda, T>0$ and $\mu$ be a symmetric minimizer of $\Eund_{\lambda,T}$. Then $\mu_T$ is atomic with finitely many non-zero atoms.
\end{theorem}

\begin{remark}
 This is in strong contrast with the case of a penalization by a $H^{-s}$ norm in dimension two. In that case a direct adaptation of Theorem  \ref{thm:main1} and Theorem \ref{thm:main2} gives that $\mu_T$ does not have atoms for $s<3/4$. A refinement of this argument shows that there are actually no atoms  for $s<1$ (alternatively this would follow from Theorem  \ref{thm:main1} provided we can show \eqref{hypIepsintro} for the optimal decay rate $\beta=2s/(2s+1)$ when $s>1/2$, giving $\dim \mu_T=1-2s$).
\end{remark}
As already observed in \cite{G}, in this two-dimensional setting, which is actually closer in spirit to a time-dependent one-dimensional setting, the no-loop property of $\mu$ implies that  the  transport map from $\mu_0$ to $\mu_T$ induced by $\mu$ is monotone (see Lemma \ref{lem:monotone}). In particular it is (up to a translation) the optimal transport map.
In order to prove Theorem \ref{thm:was1} we first show in Theorem \ref{barycenter}, that as in usual branched transport (see e.g. \cite{pegLag}), there is an equivalent  Lagrangian formulation of the problem. As opposed to standard representations in terms of probabilities on the space of curves, we parametrize the curves by their initial point. This may be seen as the analog of optimal transport plans versus optimal transport maps in optimal transport problems. The main advantage of working with a Wasserstein penalization is that in the Lagrangian viewpoint the internal kinetic energy and the boundary energy can be combined together. The problem then becomes an irrigation problem to (and from) the Lebesgue measure but where the perimeter is penalized only for part of the time.  As a consequence of this and of the no-loop property we can prove in Proposition \ref{propcone}, the so-called \textit{cone property}, which is pivotal in our analysis. This property states that each subtree starting from a Dirac mass at a point $X$ and irrigating the Lebesgue measure on an interval $A$ is contained in the cone in time-space spannned by $X$ and $A$. Using the cone property we prove in Proposition \ref{prop:refineddirac} a lower bound on the possible mass of a subtree containing a branching point (which then implies Theorem \ref{thm:was1}). The idea is that if there is a branching point of mass $\phi$ at a distance $\eps$ from the boundary, we can replace the whole subtree by a segment. This way we gain at least $\eps$ in perimeter and the cone property guarantees that we loose at most something of the order of $\eps \phi^3$ in the kinetic energy. This would contradict minimality if $\phi$ is too small. The actual proof is  more involved since we need to obtain a lower bound on the mass of each subtree to be sure that $\mu_T$ has finitely many non-zero atoms.\\

Thanks to the quantitative lower bound from Proposition \ref{prop:refineddirac} on the flux that can carry a branching tree, we  prove that for some values of $\lambda$ and $T$ minimizers of $\Eund_{\lambda,T}$ do not have branching points at all. In these cases we can completely solve the problem.

\begin{theorem}\label{thm:was2}
	If $\lambda \min( 1, T)\ll 1$, then minimizers of $\Eund_{\lambda,T}$ are made of $N$ equidistant straight segments each carrying a mass $1/N$. Moreover, 
	\[
	 N\sim \max\lt(1,\frac{\lambda}{T}\rt)^{\frac{1}{3}}.
	\]

\end{theorem}

	\section{Notation and preliminary results}\label{sec:preliminars}
	In this section we define  the energy functional  and set some notations. With the symbols $\les$, $\ges$ we indicate estimates which hold up to a universal constant. We say that $A\sim B$ if $A\les B\les A$. We use $A\ll B$ as an assumption. It means that there exists $\eps>0$ such that provided $A\le \eps B$, the conclusion holds. We denote by $\nper{\cdot}$ the distance on $Q$ and by $|\cdot|$ the Euclidean distance so that $\nper{x}=\min_{\Z^d}|x-z|$. We write  $z\in \R^{d+1}$ as $z=(x,t)\in \R^d\times \R$. The symbol $\nabla$ always refers to derivatives in the 'horizontal' variables $x$ while $\partial_t$ denotes the derivative in the 'vertical' one.  We set $Q=(\R/\Z)^d$ to be the $d-$dimensional torus and then $Q_{T}=(-T,T)\times Q$. We let $\M^+(Q_{T})$ be the set of positive measures on $Q_T$ and $\M(Q_T;\R^d)$ the set of finite measures on $Q_T$ with values in $\R^d$.
	 \subsection{Optimal transport}
	 In this section, we set some notation regarding optimal transport. For much more on this topic we refer e.g. to \cite{Sant,PeyCut}.\\
	 If $\lambda_0$ and $\lambda_1$ are positive measures on $\R^d$ with $\lambda_1(\R^d)=\lambda_2(\R^d)$ we set as usual 
	 \[
	  W_2^2(\lambda_1,\lambda_2)=\min_{\pi}\lt\{ \int_{\R^d\times\R^d} |x-y|^2 d\pi \ : \pi_i=\lambda_i, \ i=1,2\rt\}.
	 \]
Similarly, if $\lambda_1,\lambda_2\in \M^+(Q)$ with $\lambda_1(Q)=\lambda_2(Q)$, we set 
\[
 W_{\per}^2(\lambda_1,\lambda_2)=\min_{\pi}\lt\{ \int_{\R^d\times\R^d} \nper{x-y}^2 d\pi \ : \pi_i=\lambda_i, \ i=1,2\rt\}.
\]
We recall the Benamou-Brenier, or Eulerian, formulation of optimal transport. For $\X\in \{Q,\R^d\}$ and every $a<b$ we have with $*=\per$ if $\X=Q$ and $*=2$ if $\X=\R^d$,
\begin{multline}\label{BB}
 \frac{1}{b-a}W_*^2(\lambda_0,\lambda_1)\\
 =\min_{(\mu,m)}\lt\{ \int_{(a,b)\times \X} \lt|\frac{dm}{d\mu}\rt|^2 d\mu : m\ll \mu, \, \partial_t \mu + \Div m=0 \textrm{ and } \mu_a=\lambda_1, \, \mu_b=\lambda_2\rt\}.
\end{multline}
Finally, we recall the following characterization of optimal transport maps on $Q$ when $d=1$.
\begin{proposition}
 Let $d=1$ and $\lambda_1,\lambda_2\in \M^+(Q)$ be such that $\lambda_1(Q)=\lambda_2(Q)$. If $\lambda_1$ does not contain atoms, the optimal transport plan for $W_{\per}^2(\lambda_1,\lambda_2)$ is induced by a map $\Psi$. It is the only transport map from $\lambda_1$ to $\lambda_2$ which can be extended to a monotone map on $\R$ with $\Psi-x$ periodic and 
 \[
  \int_Q (\Psi-x) d\lambda_1=0.
 \]

\end{proposition}

	\subsection{The (internal) energy functional}\label{sec:internal}
	We mostly follow the notation from \cite{CGOS}. In order to give a unified presentation for the three dimensional functional considered in Section \ref{sec:Main} and the two dimensional toy model from Section \ref{sec:wasserstein}, we consider the problem in arbitrary dimension. 
	\begin{definition}
We denote by $\calA_{T}$ the set of  pairs of measures 
$\m\in\M^+(Q_{T})$, $m\in \M(Q_{T};\R^d)$ with $m\ll\mu$, satisfying the continuity equation (in the distributional sense)
\begin{equation}\label{conteq0}
 \pa_t \m +\Div m=0 \qquad \textrm{in } Q_{T},
\end{equation}
and such that $\m= d{t}\otimes \m_{t}$ where, for a.e. $t\in(-T,T)$,
$\m_{t}=\sum_i \phi_i \delta_{X_i}$ for some $\phi_i> 0$ and $X_i\in Q$. 
We denote by $\calA_{T}^*=\{\mu: \exists m, (\mu,m)\in\calA_{T}\}$ the
set of admissible $\mu$.

We define the functional  $I:\calA_{T}\to[0,\infty]$ by
\begin{equation}\label{Imum}
 I(\mu,m)=
    \int_{-T}^T  \sum_{x'\in Q} \left(\mu_{t}(x')\right)^{\frac{d-1}{d}}  \, dt + 
   \int_{Q_{T}}
   \left|\frac{dm}{d\m} \right|^2 d\m, 
\end{equation}
and (with abuse of notation) $I:\calA_{T}^*\to[0,\infty]$ by
\begin{equation}\label{Imu}
I(\m)=\min \{ I(\mu,m)\ : \ m\ll\m, \ \pa_t \m + \Div m=0\}.
\end{equation}
For $-T\le a<t<b\le T$, we set 
\begin{equation}\label{defdotPE}
 \dot{P}(\mu,t)=\sum_{x'\in Q} \left(\mu_{t}(x')\right)^{\frac{d-1}{d}} \qquad \textrm{and } \qquad \dotEcin(\mu,t)=\int_Q \left|\frac{dm_t}{d\m_t} \right|^2 d\m_t.
\end{equation}
We then define 
\begin{equation}\label{defPE}
 P(\mu, (a,b))=\int_{a}^b \dot{P}(\mu,t) dt \qquad \textrm{and } \qquad  \Ecin(\mu,(a,b))=\int_{a}^b \dotEcin(\mu,t) dt
\end{equation}
so that (recall \eqref{defI}), $I(\mu,(a,b))=P(\mu,(a,b))+\Ecin(\mu,(a,b))$. As for $I$ we, recall \eqref{defImueps}, we write $P(\mu,\eps)=P(\mu,(T-\eps,T))$ and similarly for $\Ecin(\mu,\eps)$. Finally, if the measure $\mu$ is clear from the context we sometimes omit to write it.
\end{definition}
	\begin{remark}\label{rem:alpha}
	 The power $(d-1)/d$ in the definition \eqref{Imum} of $I$ comes from the fact that this term corresponds to a perimeter in the horizontal variables, see \cite{CGOS, G}. One can of course consider the case of  an arbitrary power $\alpha\in[0,1)$. Most of the results of the paper extend to this more general functionals. However, in order to keep the number of parameters in the model reasonable we decided to stick with this particular class. 
	\end{remark}
	Let us notice that by \eqref{BB} we have for $-T\le a<b\le T$ and every $\mu\in \calA^*_T$,
	\begin{equation}\label{Ecinlowerintern}
	 \Ecin(\mu,(a,b))\ge \frac{1}{b-a} W_{\per}^2(\mu_a,\mu_b).
	\end{equation}
For two measures $\bar \m_{\pm}\in \M^+(Q)$ with $\bar \m_+(Q)=\bar \m_-(Q)$ we  consider the problem 
\begin{equation}\label{limitProb}
 \inf\lt\{ I(\mu) \ : \ \mu_{\pm T}= \bar \m_{\pm} \rt\}.
\end{equation}
By \eqref{Ecinlowerintern}, if $I(\mu)<\infty$ then $t\to \mu_t$ is $1/2$-H\"older continuous as a curve in the Wasserstein space so that the traces are well-defined. Using a simple extension of the construction from \cite[Proposition 5.2]{CGOS} we have 
\begin{lemma}\label{lem:consblock}
 For $r, T>0$ let $Q^r=(0,r)^d$ and $Q^r_T=(-T,T)\times Q^r$. Then, for every $\bar \m_{\pm}\in \M^+(Q^r)$ with $\Phi=\bar \m_+( Q^r)=\bar \m_-( Q^r)$, there exists  a pair $(\mu,m)\in \M^+(Q^r_T)\times \M(Q^r_T; \R^d)$ with $m\ll \mu$ and such that 
 \begin{itemize}
  \item $\m= d{t}\otimes \m_{t}$ where, for a.e. $t\in(-T,T)$,
$\m_{t}=\sum_i \phi_i \delta_{X_i}$ for some $\phi_i> 0$ and $X_i\in Q^r$;
\item $\m_{\pm T}= \bar \m_{\pm}$;
  \item  $\partial_t \mu+ \Div m=0$ in $Q^r_T$, for $t\in (-T,T)$, $m_t \cdot \nu = 0$ on $\partial Q^r$ where $\nu$ is the exterior normal to $Q^r$;
  \item we have 
\begin{equation}\label{eq:estimIbuildingblock}
 I(\mu,m)\les T \Phi^{\frac{d-1}{d}}+ \frac{ r^2 \Phi}{T}.
\end{equation}
 \end{itemize}

\end{lemma}

Since Neumann boundary conditions are compatible with periodic boundary conditions, applying Lemma \ref{lem:consblock} with $r=1$ yields that the infimum in \eqref{limitProb} is finite. Moreover, arguing as in \cite[Proposition 5.5]{CGOS} we have existence of minimizers.
\begin{remark}
 Notice that if we considered as suggested in Remark \ref{rem:alpha} a more general power $\alpha$ in the definition of $I$, using a simple extention of  the branching construction from \cite[Proposition 5.2]{CGOS} we would have that \eqref{limitProb} is finite for every $\bar \m_{\pm}$ provided $\alpha>(d-2)/d$. In the language of branched transportation,  under this condition every measure can be irrigated from a Dirac mass.
\end{remark}

We now introduce the notion of subsystems which slightly generalizes   \cite[Proposition 5.7]{CGOS} (the proof being totally analogous we omit it).

	\begin{proposition}\label{prop:subsystem}
		Let  $t\in [-T,T]$ and  $(\mu, m)\in \calA_T$ with $I(\mu, m)<\infty$ and set  $v=dm/d\mu$. Then for every $\sigma\le \mu_t$, there exists a measure $\widetilde{\mu}\in \calA_T^*$ such that
		\begin{itemize}
			\item[i)] $\mu'=\mu-\widetilde{\mu}$ is a positive measure;
			\item[ii)] $\widetilde{\mu}_t=\sigma$;
			\item[iii)] $\widetilde{\mu}$ satisfies the continuity equation
			$
			\partial_t\widetilde\mu+\Div \lt(v \tilde{\mu}\rt)=0.
			$
		\end{itemize}
		In particular $(\tilde{\mu},v\tilde{\mu})\in \calA_T$ and $I(\tilde{\mu},v \tilde{\mu})\le I(\mu,m)$. Similarly $(\mu',v\mu')\in \calA_T$ with $I(\mu',v\mu')\le I(\mu,m)$.
	\end{proposition}

	In view of the previous proposition, when $\mu_t=\sum_i \phi_i \delta_{X_i}$ and $\sigma_t=\phi_i \delta_{X_i}$ for some $i$ (which is exactly the case considered in \cite[Proposition 5.7]{CGOS})  we denote $\mu^{t,i}_+=\widetilde{\mu}\restr (t,T)\times Q$
	the forward subsystem emanating from $X_i$ and
	$\mu^{t,i}_-=\widetilde{\mu}\restr (-T,t)\times Q$ the backward subsystem emanating from $X_i$. Let us recall from \cite[Lemma 5.8]{CGOS} the no-loop property of minimizers (this is actually a small extension of that lemma, using Proposition \ref{prop:subsystem} instead of \cite[Proposition 5.7]{CGOS}).
	\begin{lemma}\label{lem:nolopp}
	 Let $\mu$ be a minimizer of \eqref{limitProb} and  $t\in (-T,T)$. For $X_1$ and $X_2$ in $Q$ such that  $\mu_t(X_i)\neq 0$ for $i=1,2$, denote by $\mu^{t,i}_{\pm}$ the forward and backward subsytems emanating from $X_i$. If there exist $-T\le t_-<t<t_+\le T$  such that for $\ast\in\{\pm\}$, $ \mu^{ t_\ast,1}_{\ast}$ and $ \mu^{t_{\ast},2}_{\ast}$ are not mutually singular, then $X_1=X_2$.
	\end{lemma}

% 	\begin{lemma}\label{lem:nolopp}
% 	 Let $\mu$ be a minimizer of \eqref{limitProb} and  $t\in (-T,T)$. For $X_1$ and $X_2$ in $Q$ such that  $\mu_t(X_i)\neq 0$ for $i=1,2$, denote by $\mu^i_{\pm}$ the forward and backward subsytems emanating from $X_i$. If there exist $-T\le t_-<t<t_+\le T$ and  $X_{\pm}\in Q$ such that $ \mu^i_{\pm,t_{\pm}}(X_\pm)\neq 0$, then $X_1=X_2$.
% 	\end{lemma}

In most of the paper we will consider the case where $\bar \m_+=\bar \m_-$ in \eqref{limitProb}. In that case, by symmetrization we may concentrate on minimizers of \eqref{limitProb} which are symmetric with respect to $t=0$.
\begin{remark}\label{rem:noloopsym}
 If $\mu$ is a symmetric minimizer, applying Lemma \ref{lem:nolopp} with $t_{\pm}=\pm T$ we find that for every $t\ge 0$, if $\mu_t=\sum_i \phi_i \delta_{X_i}$ with $X_i\neq X_j$ for $i\neq j$, then denoting by $\mu^{t,i}$ the forward subsystems starting from $X_i$, $\mu^{t,i}_T$ and $\mu^{t,j}_T$ must be mutually singular.
\end{remark}
Arguing as in \cite[Proposition 5.10\&  Proposition 5.11]{CGOS} we have
\begin{proposition}\label{prop:structure}
 Let $\bar \mu\in \M^+(Q)$ and $\mu$ be a symmetric minimizer of \eqref{limitProb} with $\bar \m_{\pm}=\bar \m$. Then,
 for every $T'<T$, $\mu$ is a finite graph  in $Q_{T'}$ without loops. In particular we have
 \begin{equation}\label{reprI}
  I(\mu)=\int_{-T}^{T} \sum_i \lt[ \phi_i^{\frac{d-1}{d}} + \phi_i |\dot{X_i}|^2\rt]  dt
 \end{equation}
where the sum is locally finite and between two branching points, $t\mapsto \phi_i(t)$ is constant while  $t\mapsto X_i(t)$ is affine. From now on we will always use the representation \eqref{reprI} of $I$.

\end{proposition}
\begin{remark}\label{rem:noloop}
 Notice that from the no-loop condition and the subbaditivity of $\phi\mapsto \phi^{\frac{d-1}{d}}$ both $\sharp \spt \mu_t$ and $\dot{P}(\mu,t)=\sum_i \phi_i(t)^{\frac{d-1}{d}}$ are increasing in $(0,T)$.
\end{remark}

We end this section by proving equi-partition of the energy. For similar results in related models see e.g. \cite{KM94,BeJuOt,CDKZ}.
\begin{proposition}\label{prop:equipar}
 For $\bar \mu\in \M^+(Q)$ let $\mu$ be a symmetric minimizer of \eqref{limitProb} with $\bar \m_{\pm}= \bar \m$. There exists $\Lambda\in \R$
 with 
 \begin{equation}\label{eq:estimeta}
  |\Lambda|\les \frac{I(\mu)}{T}
 \end{equation}
 such that recalling the definition \eqref{defdotPE}, we have for $t\in (-T,T)$
 \begin{equation}\label{eq:equiparstate}
  \dot{P}(\mu,t)= \dotEcin(\mu,t)+\Lambda. 
 \end{equation}
\end{proposition}
\begin{proof}
 As usual the proof goes through internal variations. Let $\xi\in C^\infty_c( -T,T)$. For $|\eps|$ small enough, $\xi_\eps={\rm Id} + \eps \xi$ is an increasing diffeomorphism of $(-T,T)$ into itself. We consider the competitor given by $\what{\phi}_i(t)=\phi_i(\xi_\eps^{-1}(t))$ and $\what{X}_i(t)=X_i(\xi_\eps^{-1}(t))$. Notice that by Proposition \ref{prop:structure} the number of branches in the support of $\xi_\eps^{-1}-{\rm Id}$ is finite so that this definition is not ambiguous. By minimality of $\mu$ we have 
 \begin{align*}
  I(\mu)&\le \int_{-T}^T \sum_i \lt[ (\phi_i\circ \xi_\eps^{-1})^{\frac{d-1}{d}} +  |\dot{\xi}_\eps^{-1}|^{ 2}(\phi_i\circ \xi_\eps^{-1}) |\dot{X_i}\circ \xi_\eps^{-1}|^2\rt]  dt\\
  &\stackrel{s=\xi_\eps^{-1}(t)}{=}\int_{-T}^T \sum_i \lt[\phi_i^{\frac{d-1}{d}}(1+\eps \dot{\xi})^{-1} + (1+\eps \dot{\xi}) \phi_i|\dot{X}_i|^2\rt] ds\\
  &= I(\mu) + \eps \int_{-T}^T (\dotEcin-\dot{P}) \dot{\xi} ds + O(\eps^2).
 \end{align*}
This concludes the proof of \eqref{eq:equiparstate}. Integrating \eqref{eq:equiparstate} we have 
\[
 \int_{-T}^T \dot{P} =\int_{-T}^T \dotEcin + 2 T \Lambda.
\]
Since $\dot{P},\dotEcin\ge0$ and 
\[
 \int_{-T}^T \dot{P} +\int_{-T}^T \dotEcin=I(\mu)
\]
we also get \eqref{eq:estimeta}.
\end{proof}
\begin{remark}
 Let us point out that the result also  holds  without the symmetry assumption (and for general boundary data). We only stated it in this form so that Proposition \ref{prop:structure} applies. This avoids technicalities to fully justify the computations.    
\end{remark}

\subsection{Sobolev spaces, Hausdorff dimensions and capacities}\label{sec:Sob}
For a measure $\sigma\in \M(Q)$, we denote by $(\what{\sigma}_k)_{k\in \Z^d}$ its Fourier coefficients. We then set for $\gamma\in \R$, 
\begin{equation}\label{def:Hgamma}
 \|\sigma\|_{H^{-\gamma}}^2=\sum_{k\in \Z^d} |k|^{-2\gamma} |\what{\sigma}_k|^2.
\end{equation}
Notice in particular that for  $\gamma>0$, $\|\sigma\|_{H^{-\gamma}}<\infty$  forces  $ \what{\sigma}_0=\sigma(Q)=0$.
 A crucial ingredient in the proof of Theorem \ref{thm:main1} is the following  characterization of negative Sobolev norms.
\begin{proposition}\label{prop:H-ssemigroup}
	For every $\gamma>0$ there holds
	\begin{equation}\label{car2}
		\|\sigma\|_{H^{-\gamma}}^2\sim\int_0^1 \eta^{2\gamma+1}\lt(\sum_{|k|\le \eta^{-1}} |k| |\sigma_k|^2\rt) \frac{d\eta}{\eta}.
	\end{equation}
	
\end{proposition}
\begin{proof} 
	We begin from the right-hand side and use Fubini to get 
	\begin{align*}
		\int_0^1 \eta^{2\gamma+1}\lt(\sum_{|k|\le \eta^{-1}} |k| |\sigma_k|^2\rt) \frac{d\eta}{\eta}&=\sum_k|k| |\sigma_k|^2\int_0^1 \chi_{\{\eta\le |k|^{-1}\}} \eta^{2\gamma+1} \frac{d\eta}{\eta}\\
		&= \sum_k|k| |\sigma_k|^2\int_0^{|k|^{-1}} \eta^{2\gamma} \frac{d\eta}{\eta}\\
		&\sim \sum_k|k| |k|^{-(2\gamma+1)}|\sigma_k|^2\\
		&= \sum_k |k|^{-2\gamma}|\sigma_k|^2\\
		&=\|\sigma\|_{H^{-\gamma}}^2.
	\end{align*}
\end{proof}
\begin{remark}\label{rem:semigroup}
 If we consider $\rho_\eta=\eta^{-d}\rho(\cdot/\eta)$ a standard  compactly supported convolution kernel and set $\sigma_\eta=\rho_\eta \ast \sigma$   using for instance Fourier series we see that  \eqref{car2} is equivalent to
 \[
  \|\sigma\|_{H^{-\gamma}}^2\sim\int_0^1 \eta^{2\gamma+1} \int_Q |\, |\nabla|^{\frac{1}{2}} \sigma_\eta\,|^2 dx  \frac{d\eta}{\eta}.
 \]
Here $|\nabla|^{\frac{1}{2}}$ can be defined for instance through its Fourier symbol by $\what{|\nabla|^{\frac{1}{2}} \sigma}_k=|k|^{\frac{1}{2}} \what{\sigma}_k$. This is a variant of classical semi-group characterizations of Sobolev spaces, see \cite{Triebel}. After the work \cite{Hairer} of Hairer these have gain a lot of popularity in the study of PDEs with randomness, see e.g. \cite{AKMbook,OtWe,GolHues}. To the best of our knowledge, it is however the first time that it is used in such a context.
\end{remark}
We now recall some definitions from geometric measure theory. We will need to consider both measure in $Q$ and in $\R^d$. Denoting by $\H^\alpha$ the $\alpha-$dimensional Hausdorff measure, we recall (see e.g. \cite{M}) that for $A\subset \X$ with $\X\in\{ Q,\R^d\}$ the Hausdorff dimension of $A$ is defined as
\[
 \dimX A=\sup\{ \alpha \, : \, \H^\alpha(A)>0\}=\inf\{\alpha \, : \, \H^\alpha(A)=0\}.
\]
For  measures there are many ways to define upper and lower dimensions \cite{MaMoRey}. We will use the same definition as in  \cite{BaGen}. For $\sigma\in \M^+(\X)$ we define the lower and upper dimensions as
%\begin{multline*}
\[
 \dimlX \sigma=\inf\{ \dimX A \, : \, \sigma(A)>0\}%\\
 \quad \textrm{ and } \quad  \dimuX \sigma=\inf\{ \dimX A \, : \,  \sigma(\X\backslash A)=0\}.
\]
 %\end{multline*}
If $\dimlX \sigma=\dimuX \sigma$, we denote by $\dim_{\X}\, \sigma$ this common value. Notice  of course that every $\sigma\in \M^+(Q)$ can be considered as a locally finite and periodic measure on $\R^d$. Its dimension on $Q$ and $\R^d$ coincide. When it is clear from the context, we will omit to specify $\X$ in the notation.\\
It is not hard to see that 
\[
 \dimlX \sigma=\sup\{ \alpha \, : \, \sigma \ll \H^\alpha\}\quad
\textrm{ and } \quad  \dimuX \sigma=\inf\{ \alpha \, : \,  \sigma \perp \H^\alpha\}.
\]
When $\X=\R^d$ we recall the connection between (lower) Hausdorff dimension and capacities. For $\alpha\in (0,d)$ and $\sigma\in \M^+(\R^d)$ we define the Riesz energy as
\[
 V_\alpha(\sigma)=\int_{\R^d\times\R^d} \frac{d\sigma(x) d\sigma(y)}{|x-y|^\alpha}.
\]
The $\alpha-$capacity of a set $A\subset \R^d$ is then given by
\[
 C_\alpha(A)=\sup\{ V_\alpha^{-1}(\sigma) \, : \,  \sigma\in \M^+(A), \,  \sigma(A)=1\}.
\]
Finally the capacitary dimension of a set $A$ is given by
\[
 \dim_c\,  A=\sup\{ \alpha  \, : \, C_\alpha(A)>0\}=\inf \{\alpha \, : \, C_\alpha(A)=0\}.
\]
By Frostman Lemma, see \cite[Theorem 8.9]{M} we have for $A\subset \R^d$,
\[
 \dim_{\R^d} \, A=\dim_c\,  A.
\]
As a consequence we have the following bound on the lower dimension of a measure in $\R^d$.
\begin{lemma}\label{lemdim}
 Let $\sigma\in \M^+(\R^d)$ with $\sigma(\R^d)<\infty$. Then
 \begin{equation}\label{eq:characlower}
  \underline{\dim}_{\R^d}\, \sigma  {\ge} \sup\{ \alpha \, : \, V_\alpha(\sigma)<\infty\}.
 \end{equation}
\end{lemma}
\begin{proof}
 Denote by $D$ the right-hand side of \eqref{eq:characlower}. For $\eps\ll1$ let $\alpha=D-\eps$  which imples that $V_\alpha(\sigma)<\infty$. Then if $A\subset \R^d$ is such that $\sigma(A)>0$ we have $C_\alpha(A)>0$ and thus $\dim_{\R^d} \, A=\dim_c\, A\ge \alpha$. This proves that
 \[
  \underline{\dim}_{\R^d}\, \sigma \ge \alpha=D-\eps.
 \]
 Since $\eps$ is arbitray this proves $\underline{\dim}_{\R^d}\, \sigma \ge D$.

\end{proof}

\begin{remark}
	In general the inequality in \eqref{eq:characlower} is strict. Indeed, for every $\alpha\in (0,d)$ if $f\in L^1(\R^d)\backslash H^{-(d-\alpha)/2}(\R^d)$ (see \eqref{Hgamma} below and recall that $L^1(\R^d)\backslash H^{-\gamma}(\R^d)=\emptyset$ only if $\gamma>d/2$), then identifying $f$ with the measure $f dx$ we have $ \underline{\dim}_{\R^d}\, f=d$ but $V_\alpha(f)=\infty$ by  \eqref{VHgamma} below.
\end{remark}

In combination with Proposition \ref{prop:H-ssemigroup} an important ingredient  is the following counterpart of Lemma \ref{lemdim} in the periodic setting.
\begin{lemma}\label{lemdimper}
 Let $\sigma\in \M^+(Q)$ with $\sigma(Q)=1$. Then 
\begin{equation}\label{eq:characlowerper}
 \underline{\dim}_{Q}\, \sigma \ge \sup\{ \alpha \, : \, \|\sigma-1\|_{H^{-\frac{1}{2}(d-\alpha)}}<\infty\}.
\end{equation}
\end{lemma}
\begin{proof}
 The proof will be obtained as a combination of Lemma \ref{lemdim} and a localization argument. Indeed,  for $\gamma\in \R$ and $\sigma\in \M^+(\R^d)$, we define the (homogeneous) $H^{-\gamma}$ norm by
 \begin{equation}\label{Hgamma}
  \|\sigma\|^2_{H^{-\gamma}(\R^d)}=\int_{\R^d}  |\xi|^{-2\gamma} |\F(\sigma)(\xi)|^2 d\xi,
 \end{equation}
 where $\F$ is the Fourier transform. We then have for some constant $C_{d,\alpha}>0$, see \cite[Theorem 5.9]{LiebLoss},
 \begin{equation}\label{VHgamma}
  V_{\alpha}(\sigma)= C_{d,\alpha}\|\sigma\|^2_{H^{-\frac{1}{2}(d-\alpha)}(\R^d)}.
 \end{equation}
In particular by Lemma \ref{lemdim} the counterpart in $\R^d$ of \eqref{eq:characlowerper} holds true. We first claim that there exists $\chi\in C^\infty_c(B_{1/4})$  with $\chi \ges 1$ on $B_{1/8}$ and a family of complex coefficients $(\eta_k)_{k\in \Z^d}$ with $|\eta_k|\les |k|^{-(d+1-\alpha)}$ such that for every $\sigma\in \M^+(Q)$ with $\sigma(Q)=1$,
\begin{equation}\label{claimHgamma}
 \|\sigma-1\|^2_{H^{-\frac{1}{2}(d-\alpha)}}= \int_{Q\times Q} \frac{(d \sigma_x-1) (d\sigma_y-1)}{\nper{x-y}^\alpha} \chi(x-y) +\sum_{k\in \Z^d} \eta_k |\what{\sigma}_k|^2. 
\end{equation}
Without loss of generality we may assume that $\sigma\in C^\infty(Q)$. Fix first $\chi\in C^\infty_c(B_{\frac{1}{4}})$  with $\chi = 1$ on $B_{1/8}$. Let 
\[
 K(x)=\frac{\chi(x)}{\nper{x}^\alpha}\in L^1(Q).
\]
By Parseval (in $Q$) there exists a constant $C>0$ such that 
\[
 \int_{Q\times Q} \frac{(d \sigma_x-1) (d\sigma_y-1)}{\nper{x-y}^\alpha} \chi(x-y)=C \sum_{k\in \Z^d} \what{K}_k |\what{\sigma}_k|^2. 
\]
By \eqref{def:Hgamma}, in order to prove \eqref{claimHgamma}, it is therefore enough to show that there exists a constant $C_{d,\alpha}>0$ such that
\begin{equation}\label{claimHgammared}
 \lt|\what{K}_k- C_{d,\alpha} |k|^{-(d-\alpha)}\rt|\les |k|^{-(d-\alpha+1)}.
\end{equation}
Since in $B_{1/4}$, we have $\nper{x}=|x|$ we find
\[
 \what{K}_k=\int_{Q}\frac{\chi(x)}{|x|^\alpha} \exp(-i k\cdot x) dx=\int_{\R^d}\frac{\chi(x)}{|x|^\alpha} \exp(-i k\cdot x) dx=\F(\tilde{K})(k) 
\]
where for $x\in \R^d$,
\[
 \tilde{K}(x)=\frac{\chi(x)}{|x|^\alpha}.
\]
Denoting $\psi=\F(\chi)$ (which is a Schwartz function) we have by \cite[Theorem 5.9]{LiebLoss}
\[
 \F(\tilde{K})(k)= C_{d,\alpha}\int_{\R^d} \frac{\psi(\xi)}{|k-\xi|^{d-\alpha}} d\xi.
\]
We claim that 
\begin{equation}\label{claim:fourier}
 \lt|\int_{\R^d} \frac{\psi(\xi)}{|k-\xi|^{d-\alpha}} d\xi - |k|^{-(d-\alpha)} \int_{\R^d} \psi\rt|\les |k|^{-(d-\alpha+1)}.
\end{equation}
Since $\int_{\R^d} \psi =\chi(0)=1$, this would conclude the proof of \eqref{claimHgammared}. On the one hand, since $\psi$ is a Schwartz function, for every $\beta\ge 1$
\[
 \int_{B_{\frac{|k|}{2}}^c} \lt(\frac{1}{|k-\xi|^{d-\alpha}} + |k|^{-(d-\alpha)}\rt) |\psi(\xi)| d\xi\les_\beta |k|^{-\beta}.
\]

On the other hand in $B_{|k|/2}$, by Taylor expansion,
\[
 \lt||k-\xi|^{-(d-\alpha)}-|k|^{-(d-\alpha)}\rt|\les |k|^{-(d-\alpha +1)}|\xi|
\]
so that \eqref{claim:fourier} follows. \\

We now prove \eqref{eq:characlowerper}. Consider first  $\alpha$ such that $\|\sigma-1\|_{H^{-\frac{1}{2}(d-\alpha)}}<\infty$. We want to prove that
\begin{equation}\label{eq:claimcharper}
\underline{\dim}_{Q}\, \sigma \ge \alpha. 
\end{equation}
 By translational invariance it is enough to prove that
\[
 \underline{\dim}_{Q}\, \sigma\restr B_{\frac{1}{8}}\ge \alpha.
\]
By \eqref{claimHgamma}, $|\eta_k|\les |k|^{-(d+1-\alpha)}$, $K\in L^{1}(Q)$ and the fact that in $B_{1/4}$ we have $\nper{x}=|x|$ we have
\[
 \int_{Q\times Q} \frac{d \sigma_x d\sigma_y}{|x-y|^\alpha} \chi(x-y)<\infty.
\]
This implies that for every cut-off function $\tilde{\chi}$ supported in $B_{1/8}$ with $0\le \tilde{\chi}\le 1$,
\[
 V_\alpha(\tilde{\chi} \sigma)=\int_{\R^d\times\R^d} \frac{\tilde{\chi}(x) d\sigma_x \tilde{\chi}(y) d\sigma_y}{|x-y|^\alpha}<\infty.
\]
By \eqref{eq:characlower} this concludes the proof of \eqref{eq:claimcharper}.

\end{proof}
% 
% \textcolor{blue}{D}\textcolor{red}{A}\textcolor{yellow}{L}\textcolor{green}{I}\textcolor{pink}{A}
We now prove a characterization of the upper dimensions.
\begin{lemma}\label{charaup}
 Let $\sigma\in \M^+(\X)$ with $0<\sigma(\X)<+\infty$. Then 
 \begin{equation}\label{eq:characupper}
  \dimuX \sigma =\sup\{ \alpha \, : \textrm{ there exists } 0\neq\tilde{\sigma}\le \sigma \textrm{ such that } \tilde{\sigma}(B_r(x))\le r^\alpha \textrm{ for all } x\in \X, r\le 1\}.
 \end{equation}

\end{lemma}

\begin{proof}
As above let $D$ be the right-hand side of \eqref{eq:characupper}. For $\eps\ll1$ set $\alpha=\dimuX \sigma-\eps$. For every $\delta>0$ let 
 \[
  A_\delta=\{x \in \X \ : \ \limsup_{r\to 0}\frac{\sigma(B_r(x))}{r^\alpha}\ge  \delta\}.
 \]
By \cite[Theorem 2.56]{AFP} we have $\H^\alpha(A_\delta)\les \delta^{-1} \sigma(\X)$ so that 
\[
A=\cup_{\delta} A_\delta=\{ x\in \X \, : \, \limsup_{r\to 0} \frac{\sigma(B_r(x))}{r^\alpha}>0\} 
\]
 satisfies $\dimX A\le \alpha$. Since $\alpha <\dimuX \sigma$, we must have $\sigma(\X \backslash A)>0$. Therefore, there exists $\delta>0$ such that setting $F_\delta=\X\backslash A_\delta$ we have  $\sigma(F_\delta)>0$. Defining for  $M\ge 1$ the sets $F^M$ as 
 \[
  F^M=\{ x\in \X \, : \, \sigma(B_r(x))\le  M r^\alpha, \textrm{ for every } r\le 1\}
 \]
we have 
 \[
  F_\delta=\lt\{ x\in \X \, : \, \limsup_{r\to 0}\frac{\sigma(B_r(x))}{r^\alpha}\le  \delta\rt\}
  \subset \cup_{M\ge 1}  F^M.
 \]
Thus there exists $M\ge 1$ such that $\sigma(F^M)>0$. Using triangle inequality it is then not hard to see that $\tilde{\sigma}=\sigma\restr F^M\le \sigma$ satisfies 
\[
 \tilde{\sigma}(B_r(x))\les M r^\alpha \qquad \textrm{for all } x\in \X, \, r\le 1.
\]
Dividing if necessary $\tilde{\sigma}$ by a large constant this proves that $D\ge \alpha=\dimuX \sigma-\eps$. Since $\eps$ was arbitrary we get 
\[
 D\ge \dimuX \sigma.
\]
To prove the opposite inequality let $\alpha= D-\eps$. By definition of $D$, there exists $0\neq \tilde{\sigma}\le  \sigma$ such that 
\[
 \tilde{\sigma}(B_r(x))\le r^\alpha \qquad \textrm{for all } x\in \X, \, r\le 1.
\]
Since $\tilde{\sigma}\le \sigma$, if $A\subset \X$ is such that $\sigma(\X\backslash A)=0$ we also have $\tilde{\sigma}(\X\backslash A)=0$ and so $\tilde{\sigma}(A)>0$. By \cite[Theorem 2.56]{AFP} we have 
\[
 \tilde{\sigma}\les \H^\alpha\restr A
\]
and therefore $\H^\alpha(A)>0$. This implies that $\dimX A\ge \alpha$. Taking the infimum over all such sets $A$ we find $\dimuX \sigma\ge \alpha=D-\eps$. As before since $\eps$ is arbitrary we have 
\[
 \dimuX \sigma\ge D,
\]
which concludes the proof of \eqref{eq:characupper}.

\end{proof}
\begin{remark}
 Let us point out that when $\X=\R^d$, arguing as in the beginning of \cite[Chapter 8]{M}, we have similarly to \eqref{eq:characlower}
 \[
  \overline{\dim}_{\R^d} \, \sigma=\sup \lt\{ \alpha \, : \, \textrm{ there exists } 0\neq\tilde{\sigma}\ll \sigma \textrm{ such that } V_\alpha(\tilde{\sigma})<\infty\rt\}. 
 \]

\end{remark}

\section{Main functional}\label{sec:Main}
In this section we fix $d=2$. For $\lambda, T>0$, recalling the definition \eqref{Imu} (see also \eqref{reprI}) of $I$ we define the energy
\begin{equation}\label{def:ener}
 \E_{\lambda,T}(\mu)=I(\mu)+ \lambda \|\mu_{\pm T}-1\|_{H^{-\frac{1}{2}}}^2.
\end{equation}
We use here the short-hand notation
\[
 \|\mu_{\pm T}-1\|_{H^{-\frac{1}{2}}}^2=\|\mu_{- T}-1\|_{H^{-\frac{1}{2}}}^2+\|\mu_{ T}-1\|_{H^{-\frac{1}{2}}}^2.
\]
Let us notice that we implicitely used  the normalization $\mu_{\pm T} (Q)=1$. However, by scaling this comes with no loss of generality. Let us also notice that arguing as for \eqref{limitProb} we can prove existence of minimizers for $\E_{\lambda,T}$. Moreover, by symmetrization we may assume that  $\mu$ is symmetric around $t=0$ i.e. $\mu_{t}=\mu_{-t}$ for $t\in [0,T]$.
\subsection{Heuristics}\label{sec:heur}
In this section we give a boiled down version of the constructions from \cite{ChokKoOt,CoOtSer}. The main purpose is to give a  heuristic motivation of the conjecture from \cite{privateCOS} regarding $\dim \mu_{ T}$.
\begin{proposition}\label{prop:scaling}
  For every $\lambda>0$ and $ T\ge 1$, we have
  \begin{equation}\label{sclaing1}
   \min_{\mu} \E_{\lambda,T}(\mu)\les  T.
  \end{equation}
  If $0<T\le 1$ and $\lambda \gg T^2$,
\begin{equation}\label{sclaing2}
   \min_{\mu} \E_{\lambda,T}(\mu)\les \min\lt(T^{\frac{1}{3}},\lambda^{\frac{2}{7}} T^{\frac{3}{7}}\rt).
  \end{equation}
  Finally if $T\le 1$ but $\lambda\les T^2$,
  \begin{equation}\label{sclaing3}
   \min_{\mu} \E_{\lambda,T}(\mu)\les \lambda^{\frac{2}{3}} T^{-\frac{1}{3}}.
  \end{equation}

\end{proposition}

\begin{proof}
 By symmetry we make the construction in $(0,T)$. If $T\ge 1$ we can apply \eqref{eq:estimIbuildingblock} in $Q^r=Q$ with $\Phi=1$ and $\bar\mu_{\pm}$ the Lebesgue measure so that 
 \[
  \E_{\lambda,T}(\mu)\les T + \frac{1}{T}\les T.
 \]
 This proves \eqref{sclaing1}.
We may thus assume that $T\le 1$.  Let $N\in \N$ and $ N^{-1}\ge r>0$ to be chosen below. We divide $Q$ in $N^2$ cubes $\overline{Q}_i$ of sidelenght $1/N$ and center $X_i$. We let $Q_i= X_i +(-r/2,r/2)^2$.  We then use \eqref{eq:estimIbuildingblock}  (with $\Phi= N^{-2}$) inside each $Q_i$ to connect 
\[
 \mu_0=\sum_i \frac{1}{N^2} \delta_{X_i} \qquad \textrm{ to } \qquad \mu_T= \sum_i \frac{1}{N^2 r^2} \chi_{Q_i}
\]
at a cost 
\[
 I(\mu)\les  NT + \frac{r^2}{T}.
\]
When $r=N^{-1}$ we have $\mu_T=1$ so that 
\begin{equation}\label{eq:scalingunif}
 \min_\mu \E(\mu)\les \min_{N\ge 1} NT + \frac{1}{T N^2}= T^{\frac{1}{3}}.
\end{equation}
For $r \ll N^{-1}$ instead, by \cite[Lemma 4.4]{CoOtSer} (see also \cite[Proposition A.2]{ChokKoOt})
\[
 \|\mu_T-1\|_{H^{-\frac{1}{2}}}^2\les \frac{1}{r N^2}
\]
so that 
\begin{equation}\label{eq:simplifiedscaling}
 \min_\mu \E_{\lambda,T} (\mu)\les \min_{N\ge 1, r\ll N^{-1}} \lt(NT + \frac{r^2}{T}+ \lambda \frac{1}{r N^2}\rt).
\end{equation}
If we set aside the constraints $N\ge 1$ and  $r\ll N^{-1}$, we can minimize the right-hand side by minimizing for instance first in $r$ and then in $N$. This yields
\begin{equation}\label{eq:parameters}
 N\sim \lambda^{\frac{2}{7}} T^{-\frac{4}{7}}, \qquad  r\sim\lambda^{\frac{1}{7}}T^{\frac{5}{7}} \qquad \textrm{and } \qquad \E_{\lambda,T}(\mu)\les \lambda^{\frac{2}{7}}T^{\frac{3}{7}}.
\end{equation}
The condition $N\ge 1$ is satisfied provided 
\begin{equation}\label{condlambda}
 \lambda\gg T^2.
\end{equation}
The condition $r\ll N^{-1}$ is satisfied provided 
\[
 \lambda\ll T^{-\frac{1}{3}}.
\]
If $\lambda\ges T^{-1/3}$ we see that $T^{1/3}\les \lambda^{2/7}T^{3/7}$ so that provided $T\le 1$ with $\lambda\gg T^2$, we have 
\[
 \min_\mu \E_{\lambda,T} (\mu)\les \min\lt(T^{\frac{1}{3}},\lambda^{\frac{2}{7}} T^{\frac{3}{7}}\rt).
\]
This proves \eqref{sclaing2}. If instead $\lambda \les T^2$, we choose $N=1$ in \eqref{eq:simplifiedscaling} and then $r= (\lambda T)^{1/3}\ll1$ which gives an energy of the order of 
\[
 \min_\mu \E_{\lambda,T} (\mu)\les \lambda^{\frac{2}{3}} T^{-\frac{1}{3}}.
\]
This concludes the proof of \eqref{sclaing3}.
\end{proof}
\begin{remark}
 Let us point out that the most interesting regime is \eqref{sclaing2}. Indeed, \eqref{sclaing1} corresponds on the one hand  to samples which are too thick so that the details of the branching structure are not relevant. On the other hand, \eqref{sclaing3} corresponds to samples which are not large enough so that the artificial boundary conditions interfere with the geometry of the microstructure. 
\end{remark}
\begin{remark}
 Arguing as in \cite{ChokConKoOt,CoOtSer} it is possible to obtain matching lower bounds. These can also be directly derived from the interpolation estimate \cite[Proposition 1.4]{CintiOt}. See also \cite[Proposition 5.3]{CGOS} for a simple proof in the case $\lambda=\infty$. 
%  \textcolor{red}{Notice that if we could find a bound of the form
%  \[
%   \min_{\mu}  \frac{1}{T}W^2(\sum_i \phi_i \delta_{X_i},\mu)+\lambda \|\mu-1\|_{H^{-1/2}}^2\ges \lambda^\alpha T^\beta \sum_i \phi_i^\gamma
%  \]
% for the correct values of $\alpha,\beta,\gamma$, we could have
% \[
%  \min \E\ges \min_{\sum_i\phi_i=1} T \sum_i \phi_i^{1/2} +\lambda^\alpha T^\beta \sum_i \phi_i^\gamma
% \]
% which by Young would give the proof.}
\end{remark}

\begin{remark}\label{rem:conjecture}
While the constructions used in the proof of Proposition \ref{prop:scaling} are certainly not minimizers of $\E_{\lambda,T}$, it is somewhat reasonable to expect that  minimizers should have a similar structure. In particular, if we assume  that minimizers are self-similar, we should be able to compute  $\dim\, \mu_T$ by looking at these constructions. Recall that  $\dim\, A=\alpha$ if for $r$ small we can cover $A$ by $n$ balls of radius $r$ in such a way that $nr^\alpha\sim 1$. If we look at the non-uniform branching construction from Proposition \ref{prop:scaling}, we see from \eqref{eq:parameters} that (setting $\lambda=1$ for simplicity)
\[
 n=N^2=T^{-\frac{8}{7}} \quad \textrm{and } \quad r=T^{\frac{5}{7}} \quad \textrm{so that } \quad  n r^{\frac{8}{5}}\sim 1.
\]
This explains the conjecture from \cite{privateCOS} that $\dim\, \mu_T=8/5$. Notice moreover that since $\eps^{3/7}\ll \eps^{1/3}$ for $\eps\ll1$, we actually expect that for every fixed value of $\lambda$ and $T$ the energy scaling in the slabs $(T-\eps,T)\times Q$ should be of the order of $\eps^{3/7}$. We thus believe that close to the boundary the microstructure resembles the non-uniform branching construction and  so  $\dim_{Q}\, \mu_T=8/5$ even if we are in the global energy regime $T^{1/3}\ll \lambda^{2/7} T^{3/7}$.   
\end{remark}
\subsection{ From local energy bounds to dimensional estimates}
In this section we prove Theorem \ref{thm:main1}.
We first prove the lower bound. As already explained, thanks to Lemma \ref{lemdimper}  it reduces  to a regularity result for $\mu_T$.
\begin{theorem}\label{thm:lower}
Let $\mu$ be a symmetric minimizer of $\E_{\lambda,T}$ and assume that there exists $\beta\in(0,1)$ such that $\limsup_{\eps\to 0} \eps^{-\beta} I(\mu,\eps)<\infty$.  Then
\[
\|\mu_T-1\|_{H^{-\frac{1}{2}(2-\alpha)}}<\infty \qquad\text{ for all }\quad \alpha <f(\beta)=\frac{1+3\beta}{1+\beta}.
\]
\end{theorem}
\begin{proof}
Let $e_1$ be the first vector of the canonical basis of $\R^2$ and $\eta\in (0,1)$. Let $\eps\le 1$ to be chosen below. We are going to construct a  competitor  $\tilde{\mu}$ for $\E_{\lambda,T}$. Since $\mu$ is symmetric, we only discuss the construction in $(0,T)$. If $\mu$ is given by
\[
 \mu_t=\sum_i \phi_i(t) \delta_{X_i(t)},
\]
we set
\[
 \mu^{\pm}_t=\sum_i \phi_i(t) \delta_{X_i^\pm(t)}
\]
where
\[
 X_i^{\pm}(t)=\begin{cases}
               X_i(t) & \textrm{if } t\in [0,T-\eps]\\[8pt]
               X_i(t)\pm \frac{\eta}{\eps} (t-(T-\eps)) e_1 & \textrm{if } t\in (T-\eps,T).
              \end{cases}
\]
In particular $\mu^\pm=\mu$ in $(0, T-\eps]$ while in $(T-\eps,T)$, we have  $\dot{X}_i^{\pm}=\dot{X}_i\pm \frac{\eta}{\eps} e_1$. We then set
\[
 \tilde{\mu}=\frac{1}{2}(\mu^-+\mu^+)=\sum_i \tilde{\phi}_i \delta_{\tilde{X}_i}.
\]
By minimality of $\mu$ we have $\E_{\lambda,T}(\mu)\le \E_{\lambda,T}(\tilde{\mu})$ which after rearangement of the terms yields

	\begin{equation}\label{centralestim}
	I(\tilde{\mu})-I(\mu)\ge\lambda\lt[ \|\mu_1-1\|_{H^{-\frac{1}{2}}}^2- \|\tilde{\mu}_1-1\|_{H^{-\frac{1}{2}}}^2\rt].
	\end{equation}
	We first estimate the left-hand side. By subbaditivity of the square root and as $\dot{X}_i^{\pm}=\dot{X}_i\pm \frac{\eta}{\eps} e_1$, we have for every $t\in(T-\eps,T)$
	\[
	 \sum_i \tilde{\phi}_i^{\frac{1}{2}}\le 2^{\frac{1}{2}} \dot{P}(\mu,t) \qquad \textrm{and} \qquad \sum_i\tilde{\phi}_i |\dot{\tilde{X}_i}|^2=\dotEcin(\mu,t) +\frac{\eta^2}{\eps^2}.
	\]
Integrating and using that by hypothesis, $P(\mu,\eps)\les_\mu \eps^\beta$, we find
\begin{equation}\label{eq:difI}
 I(\tilde{\mu})-I(\mu)\les P(\mu,\eps) +\frac{\eta^2}{\eps}\les_\mu \eps^{\beta} +\frac{\eta^2}{\eps}.
\end{equation}
We now estimate the right-hand side of \eqref{centralestim}. To simplify a bit the notation we let $\sigma=\mu_T$ and $\tilde{\sigma}=\tilde{\mu}_T$. Notice that if $\tau_\eta f(x)=f(x-\eta e_1)$ we have
\[
 \tilde{\sigma}=\frac{1}{2} (\tau_{-\eta} \sigma+\tau_\eta \sigma).
\]
Therefore
\[
 \what{\tilde{\sigma}}_k=\frac{1}{2} \lt(\exp(-2i\pi  \eta k\cdot e_1)+\exp(2i\pi  \eta k\cdot e_1) \rt) \what{\sigma}_k=\cos(2\pi \eta k\cdot e_1)\what{\sigma}_k.
\]
This gives
\[
 \|\mu_T-1\|_{H^{-\frac{1}{2}}}^2- \|\tilde{\mu}_1-1\|_{H^{-\frac{1}{2}}}^2=\sum_{k\in  \Z^d} \frac{\sin^2(2\pi \eta k\cdot e_1)}{|k|}|\what{\sigma}_k|^2\ges \eta^2 \sum_{|k|\les \eta^{-1}} \frac{|k\cdot e_1|^2}{|k|} |\sigma_k|^2.
\]
Combining this with \eqref{centralestim} and \eqref{eq:difI} we obtain
\[
 \eps^{\beta} +\frac{\eta^2}{\eps}\ges_{\mu, \lambda} \eta^2 \sum_{|k|\les \eta^{-1}} \frac{|k\cdot e_1|^2}{|k|} |\sigma_k|^2.
\]
Repeating the construction but with a shear in the direction $e_2$ instead of $e_1$ and summing the estimates, we get after optimizing in $\eps$ by choosing $\eps=\eta^{2/(1+\beta)}$,
\begin{equation}\label{eq:mainestimlower}
 \eta^\frac{2\beta}{1+\beta}\ges_{\mu,\lambda} \eta^2 \sum_{|k|\les \eta^{-1}} |k| |\sigma_k|^2.
\end{equation}

By Proposition \ref{prop:H-ssemigroup},  for $\alpha\in (0,2)$,
\[
 \|\sigma-1\|_{H^{-\frac{1}{2}(2-\alpha)}}^2\les \int_0^1  \eta^{3-\alpha} \sum_{|k|\les \eta^{-1}} |k| |\sigma_k|^2 \frac{d\eta}{\eta}\stackrel{\eqref{eq:mainestimlower}}{\les_{\mu,\lambda}} \int_0^1  \eta^{f(\beta)-\alpha} \frac{d\eta}{\eta}.
\]
The right-hand side is indeed finite for every $\alpha<f(\beta)$. This concludes the proof.
\end{proof}
As a direct consequence of Theorem \ref{thm:lower} and Lemma \ref{lemdimper} we get
\begin{corollary}[Lower bound in Theorem \ref{thm:main1}]\label{cor:lower}
Let $\mu$ be a symmetric minimizer of $\E_{\lambda,T}$ and assume that there exists $\beta\in(0,1)$ such that $\limsup_{\eps\to 0} \eps^{-\beta} I(\mu,\eps)<\infty$.  Then
\[
 \underline{\dim}\, \mu_1\ge f(\beta).
\]
\end{corollary}
We now turn to the upper bound in Theorem \ref{thm:main1}.

\begin{theorem}[Upper bound in Theorem \ref{thm:main1}]\label{thm:upper}
Let $\mu$ be  such that $\limsup_{\eps\to 0} \eps^{-\beta} I(\mu,\eps)<\infty$ for some $\beta\in(0,1)$.  Then
\[
 \overline{\dim}\, \mu_T\le g(\beta)=\frac{4(1-\beta)}{1+\beta}.
\]
\end{theorem}
\begin{proof}
	The proof builds on the proof of \cite[Proposition 5.3]{CGOS} but in  the case where the target measure is not the Lebesgue measure.  Let $\alpha<\overline{\dim} \, \mu_T$. By  Lemma \ref{charaup} there exists $0\neq \sigma\le \mu_T$ such that $\sigma(B_r(x))\le r^\alpha$ for every $x\in Q$ and $r\le 1$. Let $\tilde{\mu}\le \mu$ be given by Proposition \ref{prop:subsystem} and such that $\tilde{\mu}_T=\sigma$. By construction we have $I(\tilde{\mu},\eps)\le I(\mu,\eps)$. Notice also that since $\mu$ does not contain loops, also $\tilde{\mu}$ does not contain loops. 
 For $\eps>0$, denote
\[
 \tilde{\mu}_\eps=\sum_i \phi_i \delta_{X_i}.
\]
Then, using Young's inequality, and since $(0,T)\ni t\mapsto P(\widetilde{\mu},t)$ is increasing by the no-loop condition,
\begin{equation}\label{almostupper}
 I(\tilde{\mu},\eps)\ge \eps\sum_i \phi_i^{\frac{1}{2}} +\frac{W^2_{\per}(\tilde{\mu}_\eps, \sigma)}{\eps}\ges \eps^{\frac{4-\alpha}{4+\alpha}} \lt(\sum_i \phi_i^{\frac{1}{2}}\rt)^{\frac{4}{4+\alpha}} \lt(W^2_{\per}(\tilde{\mu}_\eps, \sigma)\rt)^{\frac{\alpha}{4+\alpha}}.
\end{equation}
We claim that
\begin{equation}\label{claimW2}
 W^2_{\per}(\tilde{\mu}_\eps, \sigma)\ges \sum_i \phi_i^{\frac{\alpha+2}{\alpha}}.
\end{equation}
In this case, since
%Indeed, provided \eqref{claimW2} holds, using that
\[
 \frac{1}{2} \frac{4}{4+\alpha} + \frac{\alpha+2}{\alpha} \frac{\alpha}{4+\alpha}=1
\]
by H\"older inequality we would get
\[
 \lt(\sum_i \phi_i^{\frac{1}{2}}\rt)^{\frac{4}{4+\alpha}} \lt(W^2_{\per}(\tilde{\mu}_\eps, \sigma)\rt)^{\frac{\alpha}{4+\alpha}}\stackrel{\eqref{claimW2}}{\ges} \lt(\sum_i \phi_i^{\frac{1}{2}}\rt)^{\frac{4}{4+\alpha}} \lt(\sum_i {\phi}^{\frac{\alpha+2}{\alpha}}\rt)^{\frac{\alpha}{4+\alpha}}\ge \sum_i \phi_i=\sigma(Q).
\]
Plugging this back into \eqref{almostupper} would give
\[
 \eps^\beta\ges_\mu I(\mu,\eps)\ge I(\tilde{\mu},\eps)\ge  \eps^{\frac{4-\alpha}{4+\alpha}} \sigma(Q)
\]
and thus $\beta\le (4-\alpha)/(4+\alpha)$ i.e. $\alpha\le g(\beta)$.\\

We are thus left with the proof of \eqref{claimW2}.  We have
\begin{equation}\label{quantprob}
 W^2_{\per}(\widetilde{\mu_\eps}, \sigma)= \min_{\sum_i \sigma_i=\sigma, \sigma_i(Q)=\phi_i} \sum_i\int_{Q} \nper{x-X_i}^2 d\sigma_i.
\end{equation}

We fix $i$ and assume up to translation that $X_i=0$.   Let $r$ be such that $\sigma_i(B_r)=\phi_i$ (notice that since $\sigma_i\le \sigma$, $\sigma_i(B_r)$ is a continuous function of $r$) and let $r_i= 2^{-1/\alpha}\phi_i^{1/\alpha}$ so that $\sigma_i(B_{r_i})\le r_i^\alpha\le  \phi_i/2$. We thus have $\sigma_i(B_r\backslash B_{r_i})\ge \phi_i/2$ and then
\[
 \int_{Q} \nper{x}^2 d\sigma_i\ge \int_{B_r\backslash B_{r_i}} r_i^2d\sigma_i\ges \phi_i^{\frac{2}{\alpha}+1}.
\]
Summing over $i$ concludes the proof of \eqref{claimW2}.
	
\end{proof}		
		
\begin{remark}\label{rem:quant}
 Notice that given $\sigma$ and $\Phi=\{\phi_i\}_{i=1}^N$, \eqref{quantprob} can be seen as a (non-standard) quantization problem
 \[
  R_\sigma(\Phi)=\min_{X_i} W^2_{\per}\lt(\sum_i \phi_i \delta_{X_i}, \sigma\rt)
 \]
and \eqref{claimW2} can be rewritten as
\begin{equation}\label{eq:quantimprove}
  R_\sigma(\Phi)\ges \sum_i \phi_i^{1+\frac{2}{\alpha}}.
\end{equation}
If we keep the number $N$ of points fixed and minimize further over $\Phi$ we find the classical quantization problem, see e.g. \cite[Lemma 3.4]{GrafLushgy},
\[
 R_{\sigma,N}=\min_{\Phi} R_\sigma(\Phi).
\]
Recalling the definition \cite[Definition 11.1]{GrafLushgy} of the lower quantization dimension of $\sigma$ as
\[
 \underline{D}_2(\sigma)=\liminf_{N\to \infty} \frac{2\log N}{- \log R_{\sigma,N}}
\]
we have by \cite[Theorem 11.5]{GrafLushgy}
\begin{equation}\label{quantdim}
 \overline{\dim} \, \sigma\le \underline{D}_2(\sigma).
\end{equation}
Equivalently, for every $\alpha<\overline{\dim}\, \sigma$,
\[
 R_{\sigma,N} \ges N^{-\frac{2}{\alpha}}.
\]
Since this estimate is also implied by \eqref{eq:quantimprove}, we can see \eqref{eq:quantimprove} as a refinement of \eqref{quantdim}.

\end{remark}

\subsection{Local energy bounds: Proof of Theorem \ref{thm:main2}}

%%%%%%%%%%%%%%%%%%%%%%%%%%%%%%%%%%%%%%%%%%%%%

The main result of this section is that we can connect any two given probability measures $\bar \mu_{\pm}$ at small cost. 
\begin{proposition}\label{prop:estimcons}
 Let $T>0$ and $\bar \mu_{\pm}$ be probability measures on $Q$. Then,  for every $\eta\in(0,1)$,
 \begin{equation}\label{eq:toprovelocalscaling}
  \min \{I(\mu) \ : \ \mu_{\pm T}=\bar \mu_{\pm}\}\le \frac{1+\eta}{2T} W_{\per}^2(\bar{\mu}_-,\bar{\mu}_+) + \frac{C}{\eta} \max\lt(T,T^{\frac{1}{3}}\rt).
 \end{equation}

\end{proposition}
\begin{proof}
 If $T\ges 1$ we can directly appeal to Lemma \ref{lem:consblock} with $r=\Phi=1$ so that we may assume without loss of generality that $T\ll 1$.
 Let $\pi$ be an optimal transport plan between $\bar \mu_-$ and $\bar \mu_+$. For $t\in [-T,T]$, we define the interpolated measure $\mu^t$ by
\[
 \int_Q \xi d\mu^t=\int_{Q\times Q} \xi\left( \frac{T-t}{2T} x+\frac{t+T}{2T}  y \right)\,d\pi(x,y) \qquad \forall \xi\in C^0(Q).
		\]
By definition we have for $-T\le t\le s\le T$
\begin{equation}\label{linearW2}
 W_{\per}^2(\mu^t,\mu^s)=\left( \frac{t-s}{2T}\right)^2 W_{\per}^2(\bar \mu_-,\bar \mu_+).
\end{equation}
Since for every $\mu$, recall \eqref{Ecinlowerintern},
\begin{equation}\label{eq:lowerEcin}
 E_{\rm cin}(\mu)\ge \frac{1}{2T} W_{\per}^2(\bar{\mu}_-,\bar{\mu}_+)
\end{equation}
 with equality for $\mu^t$, the idea is to construct a competitor $\mu$ as a discretized version of $\mu^t$. That is we want to be as close as possible to the equality in \eqref{eq:lowerEcin} while keeping a good control on the perimeter (which would be infinite for $\mu^t$). The construction is similar in $(-T,0]$ and in $[0,T)$ hence we only describe it in $[0,T)$. Notice that of course the two constructions coincide at $t=0$. We claim that we can construct $\mu$ on $[0,T)\times Q$ such that $\mu_T=\bar \mu_+$ and  for every $\eta\in (0,1)$, 
 \begin{equation}\label{claimcons}
  E_{\rm cin}(\mu, (0,T))\le \frac{1+\eta}{4T}W_{\per}^2(\bar \mu_-,\bar \mu_+) +\frac{C}{\eta} T^{\frac{1}{3}} \qquad \textrm{and } \qquad P(\mu, (0,T))\les T^{\frac{1}{3}}.
 \end{equation}
This would conclude the proof of \eqref{eq:toprovelocalscaling}.\\

 We  consider a dyadic decomposition of the time interval $[0,T]$ in intervals of the form $[t_k,t_{k+1}]$ and for each $t_k$ we will discretize $\mu^{t_k}$ on grids which refine as $k\to \infty$ (and $t_k\to T$). In $(t_k,t_{k+1})$ we simply interpolate using a sparse optimal transport plan.\\
We now give the details of the construction. 
We first define the $t_k$. For this we fix a  parameter
		\begin{equation}\label{choicet}
			\frac14<\delta<\frac12,
		\end{equation}
		and set $t_0=0$,
		\begin{equation}\label{choicetk}
		 t_k=t_{k-1}+(1-\delta)T\delta^{k-1} \qquad \textrm{for } k\ge 1.
		\end{equation}
We thus have
\[
 \lim_{k\to \infty}t_k= (1-\delta) T\sum_{k=0}^\infty \delta^k=T.
\]
 For $k\ge 0$, we now define $\mu_{t_k}$. 
 We pass not to construct the measures $\mu_t$; we begin by defining it on $t_k$ and then we suitably interpolate the $\mu_{t_k}$ via optimal transport plans.
 At time $t_0=0$ we want a surface energy of the order of $T^{-2/3}$ and thus around $T^{-4/3}$ Dirac masses (see \eqref{eq:scalingunif}). This suggests to choose $C\in (1/2,3/2)$ such that $N_0=C T^{-2/3}\in \N$ (recall that $T\ll1$ so such a choice of $C$ is possible) and  then set $r_0= N_0^{-1}$,
\begin{equation}\label{choiceNk}
 N_k= 2 N_{k-1}=2^k N_0 \qquad \textrm{and } \qquad r_k=N_k^{-1} \qquad \textrm{for } k\ge 1.
\end{equation}
For every $k\ge 0$ we then divide $Q$ in $N_k^2$ cubes $Q_{i,k}$ of sidelenght $r_k$. Denote also by $X_{i,k}$  the center of the cube $Q_{i,k}$, and $\varphi_{i,k}=\mu^{t_k}(Q_{i,k})$. We then set
		\[
		\mu_{t_k}=\sum_{i=1}^{N_k^2} \varphi_{i,k}\delta_{X_{i,k}} \qquad \textrm{for } k\ge 0.
		\]
 We complete now the construction of $\mu_t$ inside $(t_k,t_{k+1})$ for $k\ge 0$. We first recall  that if we consider two discrete measures $\lambda_0=\sum_{i=1}^n \phi_i \delta_{X_i}$ and $\lambda_1=\sum_{j=1}^m \psi_j \delta_{Y_j}$,
\[
 W_{\per}^2(\lambda_0,\lambda_1)=\min_{\Pi\in \R_+^{n\times m}}\lt\{ \sum_{ij} \Pi_{ij} \nper{X_i-Y_j}^2 \ :\  \sum_j \Pi_{ij}=\phi_i,\, \sum_i \Pi_{ij}= \psi_j\rt\},
\]
which is a linear programming problem with $n+m$ equality constraints. Therefore, while solutions can have up to  $n\times m$ non-zero entries, the extremal solutions have only at most $n+m$ (actually $n+m-1$ by  \cite[Proposition 3.4]{PeyCut}) non-zero entries. We call such a solution sparse. For $k\ge 0$, let $\pi^k$ be a sparse optimal transport plan between $\mu_{t_k}$ and $\mu_{t_{k+1}}$. {\color{red} See figure \ref{DGRfig3}.}
\begin{figure}\label{DGRfig3}
	\includegraphics[scale=.6]{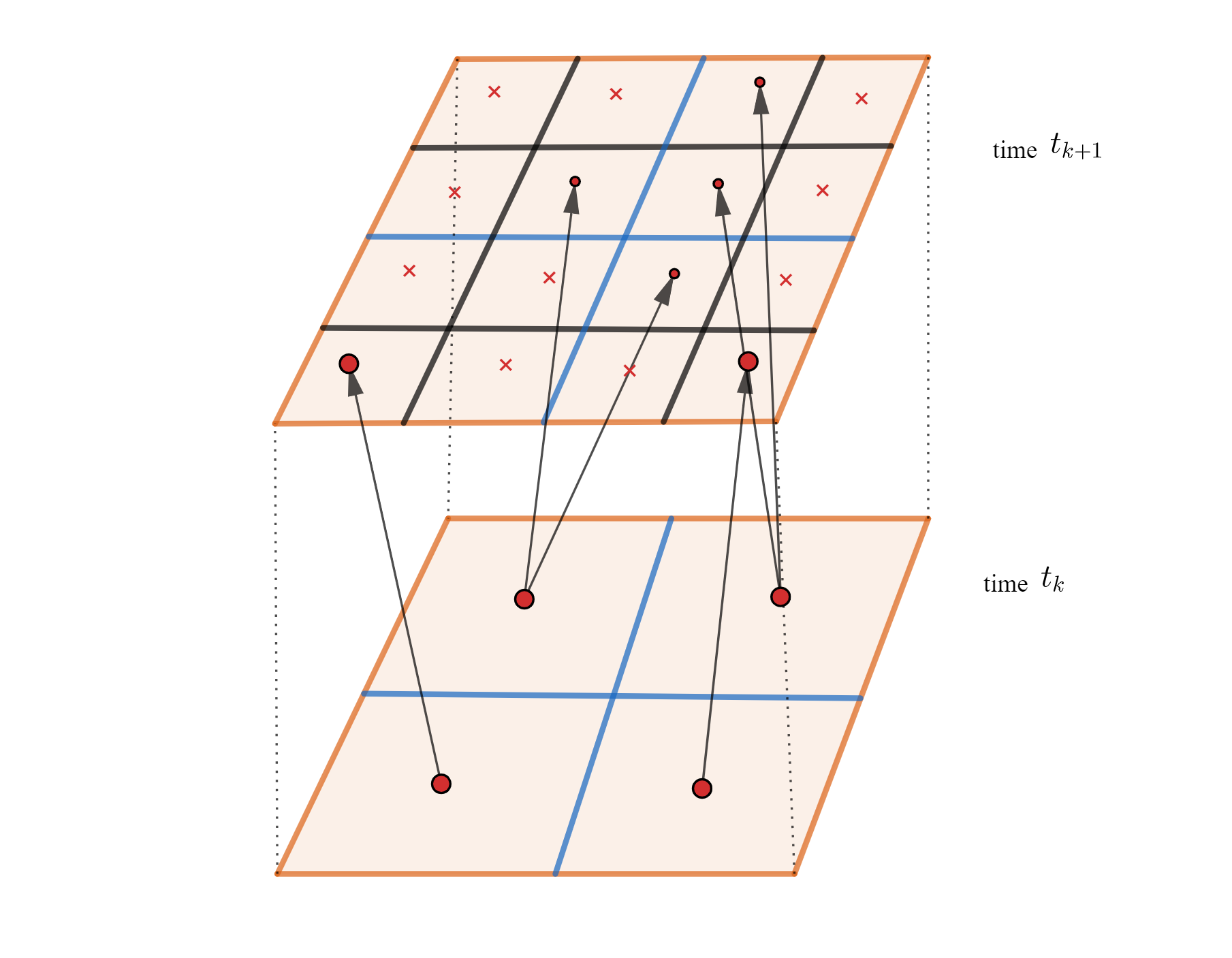}
	\caption{The idea is to connect the discrete measure at level $t_k$ with that at level $t_{k+1}$ with a sparse transport, i.e. with the least possible connections, in order to minimize the perimeter}.
\end{figure}

 We then let $\mu_t$ be the displacement interpolated measure between $\mu_{t_k}$ and $\mu_{t_{k+1}}$ through $\pi^k$ i.e.
\[
 \int_Q \xi d\mu_t=\int_{Q\times Q} \xi\left( \frac{t_{k+1}-t}{t_{k+1}-t_k} x+\frac{t-t_k}{t_{k+1}-t_k}  y \right)\,d\pi^k(x,y) \qquad \forall \xi\in C^0(Q).
\]
Notice that since $\mu_{t_k}$ is discrete for every $k\ge 0$, also $\mu_t$ is discrete  and we can write it as
\[
 \mu_t=\sum_i \phi_i \delta_{X_i}.
\]
Of course with the notation above, for every $\ell$ we have  $\phi_\ell=\pi^{k}_{ij}$  for some $i,j$ and $X_\ell$ is the linear interpolation between $ X_{i,k}$ and $X_{j,k+1}$.\\

We start by estimating the kinetic energy. By construction we have
\[
 E_k=\int_{t_{k}}^{t_{k+1}} \sum_i \phi_i |\dot{X}_i|^2=\frac{1}{t_{k+1}-t_k} W_{\per}^2(\mu_{t_k}, \mu_{t_{k+1}}).
\]
Now, by triangular and Young inequalities we have
		\[
		\begin{aligned}
			W_{\per}^2(\mu_{t_{k}},\mu_{t_{k+1}})&\le \left(W_{\per}( \mu_{t_{k}},\mu^{t_{k}})+W_{\per}( \mu^{t_{k}},\mu^{t_{k+1}})+W_{\per}( \mu^{t_{k+1}},\mu_{t_{k+1}})\right)^2\\
			&\le (1+\eta)W_{\per}^2( \mu^{t_{k}},\mu^{t_{k+1}})+\frac{C}{\eta}\left( W_{\per}^2( \mu_{t_{k}},\mu^{t_{k}})+W_{\per}^2(\mu_{t_{k+1}},\mu^{t_{k+1}})\right).
		\end{aligned}
		\]
		We have for $k\ge0$ by definition of $\mu_{t_k}$, (notice that all the masse of a cube $Q_{i,k}$ goes to the center of the cube)
		\begin{equation}\label{Wperrk}
		W_{\per}^2(\mu_{t_k},\mu^{t_k})\lesssim r_k^2.
		\end{equation}
 Using \eqref{linearW2}, we find
\begin{multline*}
 E_{\rm cin}(\mu,(0,T))=\sum_{k\ge 0} E_k\le (1+\eta)\sum_{k\ge 0} \frac{t_{k+1}-t_k}{(2T)^2}W_{\per}^2(\bar \mu_-,\bar \mu_+) +\frac{2C}{\eta}\sum_{k\ge 0} \frac{r_k^2}{t_{k+1}-t_k}\\
 =\frac{1+\eta}{4T}W_{\per}^2(\bar \mu_-,\bar \mu_+) +\frac{2C}{\eta}\sum_{k\ge 0} \frac{r_k^2}{t_{k+1}-t_k}.
\end{multline*}

Recalling the choice of $t_k$ in \eqref{choicetk}, the choice of $r_k$ in \eqref{choiceNk}  and the choice of $\delta$ in \eqref{choicet} we get 
\[
\sum_{k\ge 0} \frac{r_k^2}{t_{k+1}-t_k}\les  \sum_{k\ge 0} \frac{4^{-k} T^{\frac{4}{3}}}{T \delta^k}\\
 \les T^{\frac{1}{3}}.
\]
Therefore,
\begin{equation}\label{estimcin}
 E_{\rm cin}(\mu,(0,T))\le  \frac{1+\eta}{4T}W_{\per}^2(\bar \mu_-,\bar \mu_+) +\frac{C}{\eta} T^{\frac{1}{3}}.
\end{equation}

We now estimate the perimeter term $P(\mu,(0,T))$. We first notice that since $\pi^k$ are sparse optimal transport plans, if we set 
\[
 M(t)=\sharp\, \spt \mu_t
\]
then we have for every $k\ge 0$ and every $t\in[t_{k},t_{k+1}]$,
\begin{equation}\label{estimM}
 M(t)\le M(t_k)+M(t_{k+1}).
\end{equation}
Notice also that $M(t)$ is actually constant in each $(t_{k},t_{k+1})$. Moreover, since for $k\ge 0$, $M(t_k)=N_k^2$ we have by \eqref{choiceNk}
\begin{equation}\label{estimMin}
 M(t_k)\les  4^k T^{-\frac{4}{3}} \qquad \textrm{for all } k\ge 0.
\end{equation}
 By H\"older inequality we have for $k\ge 0$ and $t\in[t_{k},t_{k+1}]$,
\[
 \dot{P}(\mu,t)=\sum_{i=1}^{M(t)} \phi_i^{\frac{1}{2}}\le M(t)^{\frac{1}{2}}\stackrel{\eqref{estimM}\&\eqref{estimMin}}{\les} 2^k T^{-\frac{2}{3}}. 
\]
Integrating we find
\begin{equation}\label{tildeSin}
 P(\mu,(0,T))=\int_{0}^T \dot{P}(\mu,t) \les \sum_{k\ge 0} (t_{k+1}-t_k) 2^k T^{-\frac{2}{3}}\stackrel{\eqref{choicetk}}{\les} T^{\frac{1}{3}} \sum_{k\ge 0} \delta^k 2^k \stackrel{\eqref{choicet}}{\les}T^{\frac{1}{3}}.
\end{equation}
Combining this with \eqref{estimcin} concludes the proof of \eqref{claimcons}. We finally prove that $\mu_T=\bar \mu_+$. If $t\in (t_k,t_{k+1})$, by triangle inequality
\[
 W_{\per}(\mu_{t},\bar \mu_+)\le W_{\per}(\mu_{t},\mu_{t_{k+1}})+W_{\per}(\mu_{t_{k+1}},\mu^{t_{k+1}})+W_{\per}(\mu^{t_{k+1}},\bar \mu_+).
\]
Using again triangle inequality we have 
\[
 W_{\per}(\mu_{t},\mu_{t_{k+1}})\le W_{\per}(\mu_{t_k},\mu_{t_{k+1}})\le W_{\per}(\mu_{t_k},\mu^{t_k})+W_{\per}(\mu^{t_k},\mu^{t_{k+1}})+W_{\per}(\mu^{t_{k+1}},\mu_{t_{k+1}}).
\]
Thus
\begin{multline*}
 W_{\per}(\mu_{t},\bar \mu_+)\les W_{\per}(\mu_{t_k},\mu^{t_k})+W_{\per}(\mu^{t_k},\mu^{t_{k+1}}) +W_{\per}(\mu_{t_{k+1}},\mu^{t_{k+1}})+W_{\per}(\mu^{t_{k+1}},\bar \mu_+)\\
 \stackrel{\eqref{linearW2}\&\eqref{Wperrk}}{\les} r_k +\lt(\frac{t_{k+1}-t_k}{T}+ \frac{T-t_{k+1}}{T}\rt)W_{\per}(\bar \mu_-,\bar \mu_+).
\end{multline*}
This concludes the proof of $\mu_T=\bar \mu_+$ since $r_k+(t_{k+1}-t_k)+(T-t_{k+1})\to0$ as $k\to\infty$.

\end{proof}

We may now prove Theorem \ref{thm:main2}.
\begin{proof}[Proof of Theorem \ref{thm:main2}]
 Let $\bar \mu$ be a probability measure on $Q$ and $\mu$ be a symmetric minimizer of \eqref{limitProb} with $\bar \mu_{\pm}=\bar \mu$. For $\eps<T$, we apply Proposition \ref{prop:estimcons} with $\eta=1/2$, $T'=2\eps$, $\bar \mu_{-}'=\mu_{T-\eps}$ and $\bar\mu_+'=\mu_T$ to obtain by minimality of $\mu$,
 \[
  E_{\rm cin}(\eps)+P(\eps)\le \frac{3}{2\eps} W_{\per}^2(\mu_{T-\eps},\bar{\mu}) + C \max\lt(\eps,\eps^{\frac{1}{3}}\rt).
 \]
By \eqref{eq:estimeta} and \eqref{eq:equiparstate}, there exists $\Lambda$ with $|\Lambda|\les I(\mu)/T$ such that 
\begin{equation}\label{equiparthm2}
 P(\eps)=\frac{1}{2} P(\eps)+\frac{1}{2} E_{\rm cin}(\eps) +\Lambda \eps.
\end{equation}
We thus have 
 \[
  \frac{3}{2}E_{\rm cin}(\eps)+\frac{1}{2}P(\eps)\le \frac{3}{2\eps} W_{\per}^2(\mu_{T-\eps},\bar \mu) + C \max\lt(\eps,\eps^{\frac{1}{3}}\rt) + \Lambda \eps.
 \]
 Using once more that 
 \[
  E_{\rm cin}(\eps)\ge \frac{1}{\eps}W_{\per}^2(\mu_{T-\eps},\bar \mu)
 \]
we conclude after simplification that 
\[
 P(\eps)\les \max\lt(\eps,\eps^{\frac{1}{3}}\rt) + I(\mu) \frac{\eps}{T}.
\]
Plugging this into \eqref{equiparthm2} yields the same estimate for $E_{\rm cin}(\eps)$ and thus also for $I(\mu,\eps)$.
\end{proof}

%%%%%%%%%%%%%%%%%%%%%%%%%%%%%%%%%%%%%%%
%%%%%%%%%        Wasserstein 	%%%%%%%%%%%%%%%%%%%%%%%%%%%%%
%%%%%%%%%%%%%%%%%%%%%%ù
%%%%%%%%%%%%%%%%%%%%%%%%%%%%%%%%%%%%%%%%%%%%%%%%%

\section{Penalization with Wasserstein distance}\label{sec:wasserstein}
In this section we work in $\R^2$ i.e. $d=1$ and consider  for $\mu \in \calA_T^*$ the functional 
\[
 \Eund_{\lambda,T}(\mu)=I(\mu)+\lambda W_{\per}^2(\mu_{\pm T},1)=\int_{-T}^T \sum_{i=1}^{N(t)} \lt[1+ \phi_i |\dot{X}_i|^2\rt] dt +\lambda W_{\per}^2(\mu_{\pm T},1).
\]
As above it is possible to prove that there exists a symmetric minimizer of $\Eund_{\lambda,T}$.
When $d=1$, as a simple consequence of the no-loop condition the coupling between $\mu_T$ and $\mu_t$ induced by the subsystems of $\mu$ is monotone.
\begin{lemma}\label{lem:monotone}
 Let $\mu\in \calA^*_T$ be a symmetric minimizer of $\Eund_{\lambda,T}$. For every $t\in [0,T)$, if $\mu_t=\sum_i \phi_i \delta_{X_i}$ and if $\mu^{t,i}$ are the forward subsystems starting from $X_i$, then considering $\mu_t$ as a periodic measure on $\mathbb R$ we have
 \[ X_i < X_j \quad \implies  \quad \sup\{\spt \mu^{t,i}_T\}\le \inf\{\spt \mu^{t,j}_T\}.\]
Moreover for $i\neq j$, for all $x\in Q$, $\mu^{t,i}_T(\{ x\})\mu^{t,j}_T(\{x\})=0$. As a consequence, there exists a monotone map $\Psi_t: \spt \mu_T\to \R$ such that $\Psi_t-x$ is periodic and $\Psi_t\sharp \mu_T=\mu_t$. Finally $t\mapsto \Psi_t$ is compatible with the tree structure of $\mu$ in the sense that if $0\le t<s<T$ and $\mu_s^{t,i}=\sum \tilde{\phi}_j \delta_{\tilde{X}_j}$ then $\Psi_s(x)=\tilde{X}_j$ implies $\Psi_t(x)=X_i$.
 \end{lemma}
\begin{proof}
 The first part of the claim is a direct consequence of the no-loop condition, see Remark \ref{rem:noloopsym}. We define $\Psi_t$ as follows:
 \[
  \Psi_t(x)=\begin{cases}
             X_i &\textrm{ if } \mu_T^{t,i}(x)>0\\
             \min\{ X_i \ : x\in \spt \mu_T^{t,i}\} &\textrm{ if } \mu_T(x)=0.
            \end{cases}
 \]
By the definition of forward subsystems and the no-loop property we have that $\Psi_t$ is indeed monotone, $\Psi_t-x$ is periodic and $\Psi_t\sharp \mu_T=\mu_t$. Finally, if $0\le t<s<T$, then the forward subsystems of $\mu_s$ are given by the collection of all forward subsystems of $\mu^{t,i}_s$ which concludes the proof.
\end{proof}
\begin{remark}
 Since $\Psi_t$ is monotone and $\Psi_t-x$ is periodic, up to a translation it is the optimal transport map for $W_{\per}^2(\mu_T,\mu_t)$.
\end{remark}

We now prove the equivalence between the  Eulerian formulation and a  Lagrangian formulation of the problem. In the  Lagrangian point of view, the boundary penalization term $\lambda W^2_{\per}(\cdot,1)$ can be considered as part of the kinetic energy (see \eqref{eq:lagrb}). We first introduce some notation. For $\tau>0$ we set
\[
 \calC_\tau=\{ X: [-\tau,\tau]\times Q\to Q \ : \ X(-\tau,x)=X(\tau,x)=x, \textrm{ and for a.e. } x, \ t\mapsto X(t,x) \textrm{ is AC}\}.
\]
Here and from now on, AC stands for absolutely continuous.
For $\lambda, T>0$ we set $T_\lambda= T+ \lambda^{-1}$. We define for $X\in \calC_{T_\lambda}$, $y\in Q$ and $t\in(-T,T)$ the multiplicity function, counting the number of particles in $y$ at  time $t$:
\begin{equation}\label{defphiX}
 \phi_{X}(y,t)=|\{x\in Q \ : \ X(t,x)=y\}|.
\end{equation}
We then set 
\begin{equation}\label{eq:lagrb}
 \cL_{\lambda,T}(X)= \int_{Q_T} \frac{1}{\phi_X(X(t,x),t)} dx dt+\int_{Q_{T_\lambda}} |\partial_t X|^2 dx dt.
\end{equation}
% \[
%  \cL_{\lambda,T}(X)=\int_{Q_{T_\lambda}} |\partial_t X|^2+ \frac{\chi_{[-T,T]}(t)}{\phi_X(X(t,x),t)} dx dt.
% \]
We use here the convention that $0^{-1}=\infty$. Let us insist on the fact that in the definition of $\cL_{\lambda,T}$ the kinetic energy is computed for $|t|\le T_\lambda$ while the perimeter part is computed only for $|t|\le T<T_\lambda$.

\begin{theorem}\label{barycenter}
	For every $T,\lambda>0$,
	\[
	\min_{\mu\in \calA^*_T} \Eund_{\lambda,T}(\mu)=\min_{X\in \calC_{T_{\lambda}}} \cL_{\lambda,T}(X).
	\]
	Moreover, from every  collection of curve $X\in \calC_{T_\lambda}$ minimizing $\cL_{\lambda,T}$, letting for $t\in[-T,T]$, $\mu_t:= X(t,\cdot)_\sharp 1$ we can construct a minimizer of $\Eund_{\lambda,T}$. Viceversa,  every symmetric minimizer $\mu\in \calA^*_T$ of $\Eund_{\lambda,T}$, induces a minimizing $X\in \calC_{T_\lambda}$ of $\cL_{\lambda,T}$.
\end{theorem} 
\begin{proof}
	Let $X\in \calC_{T_\lambda}$ be such that  $\cL_{\lambda,T}(X)<\infty$. Let $\mu_t:=X(t,\cdot)_\sharp 1$ for $t\in[-T,T]$. Fix $t\in (-T,T)$ such that 
	\begin{equation}\label{finitephi}
	 \int_Q \frac{1}{\phi_X(X(t,x),t)} dx<\infty.
	\end{equation}
	Notice first that there is an at most countable collection of points $X_i\in Q$ such that $\phi_{X}(X_i,t)>0$. Moreover, from \eqref{finitephi}, for a.e. $x\in Q$, $\phi_X(X(x,t),t)>0$ and thus $X(x,t)=X_i$ for some $i$. Letting 
	\[
	 A_i=\{ x\in Q \ : \ X(x,t)=X_i\} \qquad \textrm{and} \qquad \phi_i=|A_i|
	\]
we have $\phi_X(X_i,t)=\phi_i$ and 
\[
 \int_Q \frac{1}{\phi_X(X(t,x),t)} dx=\sum_i \int_{A_i} \frac{1}{\phi_X(X(t,x),t)} dx=\sum_{\phi_i>0} 1.
\]
Therefore there are actually only $N(t)=\sharp\{\varphi_i(t)\ne0\}<\infty$ such points $X_i$ and 
\begin{equation}\label{Lagpoints}
 \int_{-T}^T \int_Q \frac{1}{\phi_X(X(t,x),t)} dx dt=\int_{-T}^T N(t) dt.
\end{equation}
We claim that $\mu_t=\sum_{i=1}^{N(t)} \phi_i \delta_{X_i}$. Indeed, for every $\xi\in C(Q)$,
\[
 \int_Q \xi d\mu_t=\int_Q  \xi(X(t,x)) dx=\sum_{i=1}^{N(t)} \int_{A_i} \xi(X_i) dx=\sum_{i=1}^{N(t)} \phi_i \xi(X_i).
\]
We then define for $t\in (-T,T)$ the measure $m_t\in \M(Q,\R)$ by its action on $\xi\in C(Q)$: 
\[
 \int_Q \xi dm_t=\int_Q \xi(X(t,x)) \partial_t X(t,x) dx.
\]
We have $m\ll \mu$ and $\partial_t \mu + \partial_x m=0$ so that $(\mu,m)\in \calA_T$. Using \cite[Proposition 5.18]{Sant}, we find
\begin{align*}
 \frac{1}{2}\int_{Q_T} \lt|\frac{dm}{d\m}\rt|^2 d\mu&=\sup_{\xi\in C(Q_T)} \int_{Q_T} \xi dm-\frac{|\xi|^2}{2} d\mu\\
 &=\sup_{\xi\in C(Q_T)} \int_{Q_T} \lt[\xi(X,t) \partial_t X -\frac{|\xi(X,t)|^2}{2}\rt] dxdt\\
 &\le \int_{Q_T} \sup_{\xi\in \R} \lt[\xi \partial_t X -\frac{|\xi|^2}{2}\rt] dxdt\\
 &=\frac{1}{2}\int_{Q_T} |\partial_t X|^2 dxdt.
\end{align*}
In combination with \eqref{Lagpoints} this gives 
\begin{equation}\label{upperILag}
  \int_{Q_T} |\partial_t X|^2+\frac{1}{\phi_X(X,t)} dx dt\ge I(\mu).
\end{equation}
We now consider the boundary part of the energy. For every $x\in Q$,
	\[
	 \int_{T}^{T_\lambda} |\partial_t X(t,x)|^2 dt\ge \frac{1}{T_\lambda-T}\nper{X(T_\lambda,x)-X(T,x)}^2=\lambda\nper{x-X(T,x)}^2.
	\]
Therefore, recalling that $\mu_T=X(T,\cdot)_\sharp 1$ we have
\[
 \int_Q \int_{T}^{T_\lambda} |\partial_t X|^2 dt dx\ge \lambda W^2_{\per}(\mu_T,1).
\]
Similarly
	\[
	 \int_Q\int_{-T_\lambda}^{-T}|\partial_t X|^2 dt dx\ge \lambda W^2_{\per}(\mu_{-T},1).
	\]
	Combining this with \eqref{upperILag} concludes the proof of 
	\[
	 \cL_{\lambda,T}(X)\ge \Eund_{\lambda,T}(\mu).
	\]

	We now prove the opposite inequality. Let $\mu\in \calA^*_T$ be a symmetric minimizer of $\Eund_{\lambda,T}$.  By symmetry, we will give the construction only for  $t\in [0,T_\lambda]$. As usual we write $\mu_t=\sum_{i=1}^{N(t)} \phi_i \delta_{X_i}$ where $N(t)$ is increasing. Let $\Psi'$ be the optimal transport map for $W_{\per}(1,\mu_T)$. We can then see $\Psi'$ as a monotone map from $\R\to \R$ with $\Psi'-x$ periodic. For $t\in [0,T)$, let $\Psi_t$ be the map given by Lemma \ref{lem:monotone}. We then set
	\begin{equation}
	 X(t,x)=\begin{cases}\label{consX}
	         \Psi_t(\Psi'(x)) & \textrm{ if } t\in [0,T)\\
	         \lambda \lt[ (t-T)x+(T_\lambda-t) \Psi'(x)\rt] & \textrm{ if } t\in [T, T_\lambda].
	        \end{cases}
	\end{equation}
By definition, $X(T_\lambda,x)=x$ and for every $t\in [0,T_\lambda]$, the map $x\mapsto X(t,x)-x$ is periodic so that  we can consider $X(t,\cdot)$ as a map from $Q$ to $Q$. We now prove that for a.e. $x\in Q$, the curves $t\mapsto X(t,x)$ are AC. To prove this it is enough to prove that for a.e. $x\in Q$, $X(t,x)$ is continuous at $t=T$, and that
\begin{equation}\label{kineticupeuler}
 \int_Q\int_0^T |\partial_t X(t,x)|^2 dt dx= \int_0^T \sum_{i=1}^{N(t)} \phi_i |\dot{X}_i|^2 dt\quad \textrm{and } \quad \int_{Q}\int_{T}^{T_\lambda} |\partial_t X(t,x)|^2 dtdx= \lambda W_{\per}^2(\mu_T,1).
\end{equation}
On the one hand, by definition of $X$, we have for a.e. $x\in Q$
\[
 \lim_{t\dw T} X(t,x)=\Psi'(x).
\]
On the other hand, for $t\in[0,T)$, since $\Psi'_\sharp 1=\mu_T$ and $(\Psi_t)_\sharp \mu_T=\mu_t$ (thus $X(t,\cdot)_\sharp 1=\mu_t$) we have 
\[
\int_{Q}\nper{X(t,x)-\Psi'(x)}^2 dx= \int_{Q}\nper{\Psi_t-x}^2 d\mu_T= W_{\per}^2(\mu_T,\mu_t)\stackrel{\eqref{Ecinlowerintern}}{\le} |t-T|I(\mu).
\]
Therefore, for a.e. $x\in Q$,
\[
 \lim_{t\up T} X(t,x)=\Psi'(x).
\]
This concludes the first part of the claim. We now prove \eqref{kineticupeuler}. Since $\Psi'$ is the optimal transport map for   $W_{\per}(1,\mu_T)$, the definition of $X$ for $t\in [T,T_\lambda]$ directly gives the second equality in the statement.
To prove the first equality we notice that  and it is enough to consider an interval $(t_-,t_+)$ on which $\mu_t$ does not have branching points and prove that
\begin{equation}\label{toproveequiv}
 \int_{t_-}^{t_+}\int_Q |\partial_t X|^2 dx dt=\int_{t_-}^{t_+} \sum_{i=1}^N  \phi_i|\dot{X}_i|^2 dt.
\end{equation}
On such an interval, by the compatibility of $\Psi_t$ with the tree structure of $\mu$ we have that for $\mu_T$ a.e. $x\in Q$, 
there exists $X_i$ such that $\Psi_t(x)=X_i$ for $t\in (t_-,t_+)$. This implies that for a.e. $x\in Q$, there exists $X_i$ such that $X(t,x)=X_i$ for $t\in (t_-,t_+)$. Since $X(t,\cdot)_\sharp 1=\mu_t$ we deduce that 
\begin{equation}\label{identif}
 \partial_t X(t,x)=\dot{X}_i \qquad \textrm{and} \qquad \phi_X(X,t)=\phi_i.
\end{equation}
This yields \eqref{toproveequiv}. Notice also  that \eqref{identif} also gives
\begin{equation}\label{identifperim}
 \int_{0}^{T}\int_Q \frac{1}{\phi_X(X,t)} dx dt=\int_{0}^{T} N(t) dt. 
\end{equation}
In conclusion, $X\in \calC_{T_\lambda}$ and by \eqref{kineticupeuler} combined with \eqref{identifperim},
\[
 \cL_{\lambda,T}(X)= \Eund_{\lambda,T}(\mu).
\]
\end{proof}
\begin{remark}\label{rem:monotone}
 {\rm 
 Let us point out that as a consequence of the no-loop condition, if $X$ is a symmetric minimizer of $\cL_{\lambda,T}$ then for every $t\in (-T_\lambda,T_\lambda)$, $X(t,\cdot)$ is a monotone map, see \eqref{consX}.
 }
\end{remark}

\begin{remark}\label{remark:barycenter}{\rm 
For $s\in[0,1]$, let us consider the barycenter problem
\begin{equation}\label{barycenteralter}
 \min_{\nu\in\M^+(Q)} (1-s)W^2_{\per}(1,\nu)+sW^2_{\per}(\nu,\mu).
\end{equation}
		Neglecting the interfacial term  in the previous discussion, we obtain a similar representation
		\[
		\min_{\nu\in\M^+(Q)} (1-s)W^2_{\per}(1,\nu)+sW^2_{\per}(\nu,\mu)=\min_{X(0,x)=x, X(1,\cdot)_\sharp 1=\mu} \int_0^1\int_Q |\partial_t X|^2 dx dt. 
		\]
		Moreover, any solution of \eqref{barycenteralter} is given by $X(s,\cdot)_\sharp 1$ where $X$ is a minimizing family of curves.
		Since 
		\[
		 \int_0^1 |\partial_t X|^2 dt\ge \nper{X(1,x)-X(0,x)}^2
		\]
we have that 
\[
 \min_{X(0,x)=x, X(1,\cdot)_\sharp 1=\mu} \int_0^1\int_Q |\partial_t X|^2 dx dt=W_{\per}^{2}(1,\mu)
\]
and if $\Psi$ is the optimal transport map for $W_{\per}^2(1,\mu)$, the optimal trajectories are given by $X(t,x)=(1-t)x+t \Psi(x)$. This shows that the solution to \eqref{barycenteralter} is given by McCann's interpolant $\nu=X(s,\cdot)_\sharp 1$. This proof of course also works in the Euclidean case and can be extended to more general starting measure than the Lebesgue measure (when optimal transport maps do not exist, one needs to work with probability measures on the space of curves, see for instance \cite[Proposition 5.31]{Sant}). This gives an alternative proof of    \cite[Section 6.2]{AC}, where the barycenter problem (for the much more complicated case of an arbitrary number of measures) has  been introduced first. }
\end{remark}

A crucial consequence of the Lagrangian formulation is what we call the {\em cone property} of minimizers. In order to state it we need to introduce some notation.  Let $X$ be a symmetric minimizer of $ \cL_{\lambda,T}$ and $\mu_t=X(t,\cdot)_\sharp1$ the corresponding minimizer of $\Eund_{\lambda,T}$. For every $t\in(-T,T)$ we can consider $\mu_t$ as a periodic measure on $\R$ (correspondingly $X-x$ as a periodic family of curves) which writes as 
\[
 \mu_t=\sum_{i\in \Z} \phi_i \delta_{X_i}
\]
and we can assume without loss of generality that the points $X_i$ are indexed in increasing order i.e. $X_i<X_{i+1}$. We then let for $t\in [0,T)$,
\begin{equation}\label{defAi}
 A_i(t)=\{ x\in \R \ : \  X(t,x)=X_i\}.
\end{equation}
Since $X(t,\cdot)$ is monotone, see Remark \ref{rem:monotone}, we have for some $x_{i,t}^\pm\in\R$
\[
 A_i(t)=[x_{i,t}^-,x_{i,t}^+) \qquad \textrm{with} \qquad x_{i,t}^+=x_{i+1,t}^-.
\]
We then set for $s\in [t,T_\lambda]$,
\begin{equation}\label{defcone}
 x_{i,t}^{\pm}(s)=\frac{s-t}{T_\lambda-t} x_{i,t}^{\pm}+ \frac{T_\lambda-s}{T_\lambda-t}X_i.
\end{equation}
\begin{remark}
 Let us point out that $X(s,\cdot)_{\sharp} \chi_{A_i(t)}$ for $s\in [t,T)$ coincides with the forward subsystem starting from $X_i(t)$.
\end{remark}

The cone property states that if $x\in A_i(t)$ then for $s\in (t,T_\lambda)$, the point $X(s,x)$ lies in the cone in time-space spanned by the vertex $(t,X_i)$ and the segment $(T_\lambda, A_i(t))$, see  Figure \ref{DGRfig2}.
\begin{figure}
	\includegraphics[scale=.9]{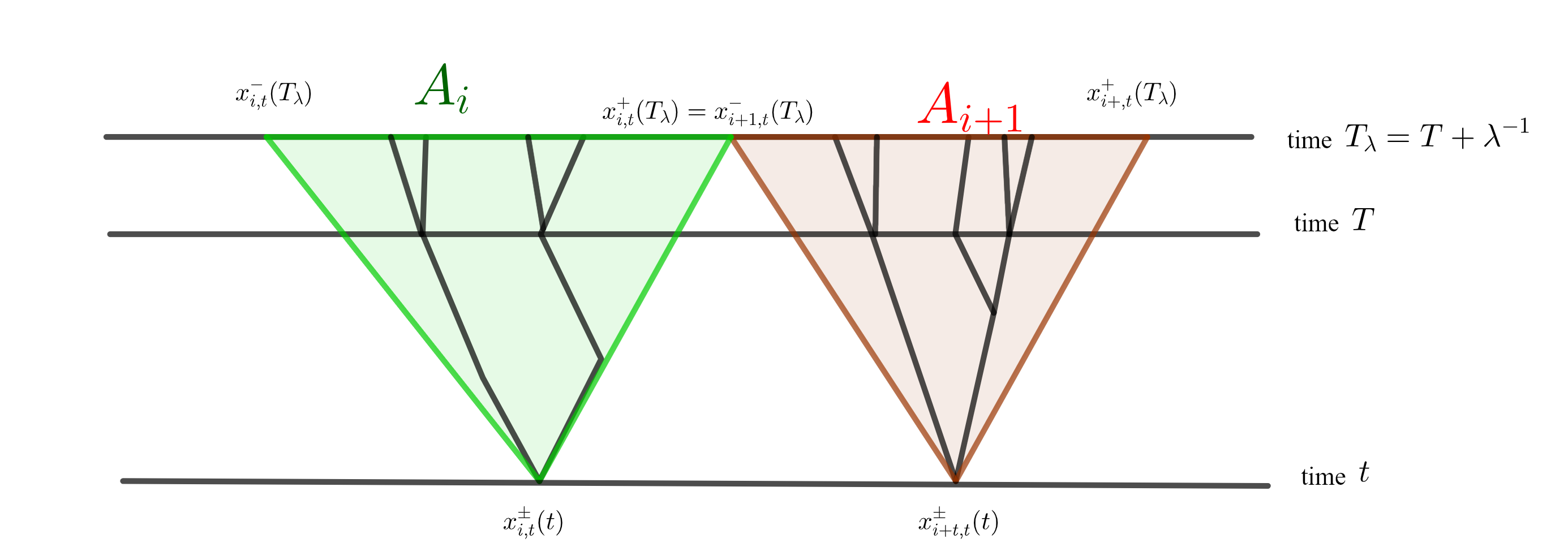}
	\caption{The cone property}\label{DGRfig2}
\end{figure}
\begin{proposition}\label{propcone}
 Let $X$ be a symmetric minimizer of $ \cL_{\lambda,T}$. Then, for every $t\in [0,T)$ and every $i\in \Z$,
 \begin{equation}\label{eq:cone}
  X(s,x)\in [x_{i,t}^{-}(s),x_{i,t}^{+}(s)] \qquad \textrm{ for every } x\in A_i(t), \ s\in [t,T_\lambda].
 \end{equation}
% Moreover, for every $i\in \Z$,
% \begin{equation}\label{barycenter}
%  A_i(0)= X_i +[-\phi_i/2,\phi_i/2).
% \end{equation}

\end{proposition}
\begin{proof}
 We fix $t\in [0,T)$ and construct a competitor by projecting the trajectories starting in $A_i(t)$ on $[x^-_{i,t}(s),x^+_{i,t}(s)]$:
 \[
  \what{X}(s,x)=\begin{cases}
                X(s,x) &\textrm{if } s\le t\\
                \max(\min( X(s,x), x_{i,t}^+(s)),x_{i,t}^-(s)) & \textrm{if } s\ge t \textrm{ and } x\in A_i(t). 
               \end{cases}
 \]
 We first notice that since $X(s,\cdot)-x$ is periodic, this construction gives also a periodic $\what{X}(s,\cdot)-x$. Furthermore, by the no-loop property this operation can only decrease the  surface part of the energy. Fix $x\in A_i(t)$. By definition of $x_{i,t}^{\pm}(s)$,  we have  $\what{X}(s,x)=X(s,x)$ for $s\in\{t, T_\lambda\}$. Moreover, since $X(x,\cdot)$ is an AC curve, we  have a partition of  $(t,T_\lambda)$ in countably many intervals as
 \[
  (t,T_\lambda)=\mathcal{I}\cup \left(\cup_j \mathcal{I}^+_j\right)\cup \left( \cup_j \mathcal{I}^-_j\right) 
 \]
where 
\[\what{X}(s,x)=\begin{cases}
                 X(s,x) &\textrm{for } s\in \mathcal{I}\\
                 x_{i,t}^{\pm}(s) &\textrm{for } s\in \mathcal{I}^{\pm}_j. 
                \end{cases}
\]
We thus have 
\[
 \int_{0}^{T_\lambda} |\partial_t \what{X}|^2 -\int_{0}^{T_{\lambda}}|\partial_t X|^2=\sum_j \lt(\int_{\mathcal{I}^+_j}  |\partial_t \what{X}|^2 -\int_{\mathcal{I}^+_j}|\partial_t X|^2\rt)
 +\sum_{j} \lt(\int_{\mathcal{I}^-_j}  |\partial_t \what{X}|^2 -\int_{\mathcal{I}^-_j}|\partial_t X|^2\rt).
\]
Let us prove that for every $j$,
\begin{equation}\label{toprovecone}
 \int_{\mathcal{I}^{\pm}_j}  |\partial_t \what{X}|^2 \le\int_{\mathcal{I}^{\pm}_j}|\partial_t X|^2
\end{equation}
with equality only if $X(s,x)= x_{i,t}^{\pm}(s)$ in $\mathcal{I}^{\pm}_j$. This would conclude the proof of \eqref{eq:cone}. For definiteness we consider an interval $\mathcal{I}^+_j=(s_1,s_2)$. Since $\what{X}(s,x)=X(s,x)$ for $s\in\{t, T_\lambda\}$ we  also have $\what{X}(s,x)=X(s,x)=x_{i,t}^+(s)$ for $s\in\{s_1,s_2\}$. Therefore,
\[
 \int_{s_1}^{s_2} |\partial_t \what{X}|^2 =\frac{|x_{i,t}^+(s_2)-x_{i,t}^+(s_1)|^2}{s_2-s_1}=\frac{|X(s_2,x)-X(s_1,x)|^2}{s_2-s_1}\le \int_{s_1}^{s_2} |\partial_t X|^2,
\]
with equality only if $X(s,x)=x_{i,t}^+(s)$ for $s\in [s_1,s_2]$.
This concludes the proof of \eqref{toprovecone}.\\

% We now turn to \eqref{barycenter}. For every $i\in \Z$, since $|A_i(0)|=\phi_i$, there exists $Y_i$ such that 
% \[
% A_i(0)= Y_i + X_i +[-\phi_i/2,\phi_i/2). 
% \]
% We then construct a competitor by setting for $x\in A_i(0)$ and $t\in[0,T_\lambda]$,
% \[
%  \what{X}(t,x)=X(t,x) + \frac{t-T_\lambda}{T_\lambda} Y_i. 
% \]
% This leaves the surface part of the energy unchanged  and using that 
% \[
% \int_{A_i(0)}\int_{0}^{T_\lambda}\partial_t X=\int_{A_i(0)} X (T_\lambda,x)-X(0,x)dx=\phi_i Y_i
% \]
% we have 
% \begin{multline*}
%  \int_{A_i(0)}\int_{0}^{T_\lambda} |\partial_t \what{X}|^2=\int_{A_i(0)}\int_{0}^{T_\lambda} |\partial_t X|^2 +\frac{2}{T_\lambda} Y_i \int_{A_i(0)}\int_{0}^{T_\lambda}\partial_t X + \frac{\phi_i}{T_\lambda}Y_i^2\\
%  =\int_{A_i(0)}\int_{0}^{T_\lambda} |\partial_t X|^2-\frac{\phi_i}{T_\lambda}Y_i^2.
% \end{multline*}
% Since $X$ is a minimizer this implies $Y_i=0$ and thus \eqref{barycenter}.
\end{proof}
\begin{remark}
 Let us point out that applying Proposition \ref{propcone} with $t=0$ and $s=T$ already guarantees that the measure $\mu_T= X(T,\cdot)_\sharp 1$ is not the Lebesgue measure.
\end{remark}

We now define the energy of connecting one dirac mass at time $0$ to the corresponding Lebesgue measure at time $T_\lambda$. For $\Phi, T>0$ and $\bar X\in \R$ we set 
\begin{multline*}
 \calC^{\bar X}_{T,\Phi}=\{ X: [0,T]\times [-\Phi/2,\Phi/2]\to \R \ : \ X(0,x)=\bar X,\  X(T,x)=x,\\
 \textrm{ and for a.e. } x, \ t\mapsto X(t,x) \textrm{ is AC}\}.
\end{multline*}
Recalling that $T_\lambda=T+ \lambda^{-1}$, we then set for $X\in \calC^{\bar X}_{T_\lambda,\Phi}$,
\begin{equation}\label{lagrangianDirac}
 \cL_{\lambda,T,\Phi}(X)=\int_{0}^T\int_{-\frac{\Phi}{2}}^{\frac{\Phi}{2}} \frac{1}{\phi_X(X(t,x),t)} dx dt+\int_{0}^{T_\lambda}\int_{-\frac{\Phi}{2}}^{\frac{\Phi}{2}} |\partial_t X|^2 dx dt,
\end{equation}
where the multiplicity $\phi_X$ is defined as in \eqref{defphiX}. We then define

\begin{equation*}%\label{def:Ephitlamb}
 E(\Phi,T,\lambda,\bar X)=\min_{X\in \calC_{T_\lambda,\Phi}^{\bar X}} \cL_{\lambda,T,\Phi}(X), \quad E(\Phi,T,\lambda)=E(\Phi,T,\lambda,0) \quad \textrm{and} \quad E(T,\lambda)=E(1,T,\lambda).
\end{equation*}
We first point out the following scaling properties of $E(\Phi,T,\lambda,\bar X)$ (see also \cite[Lemma 3.1]{G}).
\begin{lemma}\label{lem:rescalingE}
 Let $\Phi,T,\lambda,r>0$ and $\bar X\in \R$. We have on the one hand
 \begin{equation}\label{eq:movebarX}
  E(\Phi,T,\lambda,\bar X)= E(\Phi, T,\lambda) +\frac{\Phi}{T_\lambda} |\bar X|^2
 \end{equation}
and if $X$ is a minimizer for $E(\Phi, T,\lambda)$, then
\[
 \what{X}(t,x)= X(t,x) +\frac{T_\lambda-t}{T_\lambda} \bar X
\]
is a minimizer for $ E(\Phi,T,\lambda,\bar X)$. On the other hand,
 \begin{equation}\label{eq:rescalingE}
  E(\Phi,T,\lambda)=rE(r^{-\frac{2}{3}} \Phi, r^{-1}T, r\lambda)
 \end{equation}
and if $X$ is a minimizer of $E(\Phi,T,\lambda)$ then 
\[
 \what{X}(x,t)=r^{-\frac{2}{3}}X(r^{\frac{2}{3}}x, rt)
\]
is a minimizer of $E(r^{-\frac{2}{3}} \Phi, r^{-1}T, r\lambda)$. 
% As a consequence,
% \[
%  E(\Phi,T,\lambda)=\Phi^{3/2} E(\Phi^{-3/2} T, \Phi^{3/2}\lambda).
% \]
\end{lemma}
\begin{proof}
 The proof is immediate using the formulas defining $\what{X}$.
\end{proof}
 A first important consequence of Proposition \ref{propcone} is that if $\mu$ is a minimizer of $\cL_{\lambda,T}$ then every forward subsystem is a minimizer of some $E(\Phi,T,\lambda,\bar X)$.
 \begin{proposition}\label{prop:reduce}
  Let $\mu$ be a symmetric  minimizer of $\Eund_{\lambda,T}$ and let $X$ be the corresponding minimizer of $\cL_{\lambda,T}$. Fix $t\in [0,T)$ and write $\mu_t=\sum_{i} \phi_i \delta_{X_i}$. Then,  recalling the definition \eqref{defAi} of $A_i(t)$ and setting for $s\in[0,(T-t)_\lambda]$,
  \[
   \bar{X}_i=\frac{1}{\phi_i}\int_{A_i(t)} x dx, \qquad \what{X}_i(s,x)= X(s+t,x+\bar X_i)- \bar X_i,
  \]
we have that $\what{X}_i\in \calC^{X_i-\bar X_i}_{(T-t)_\lambda,\phi_i} $ is a minimizer of $E(\phi_i,T-t,\lambda, X_i-\bar{X}_i)$. As a consequence, 
\[
 \min_{\calC_{T}}\cL_{\lambda,T}(X)=2\min_{\sum_i \phi_i=1}E(\phi_i,T,\lambda)
\]
and if $\mu_0=\sum_i \phi_i \delta_{X_i}$, we have $X_i=\bar X_i$ for every $i$.
 \end{proposition}
\begin{proof}
 The first part of the statement is a direct consequence of Proposition \ref{propcone}. For the second part we have 
 \[
  \min_{\calC_{T}}\cL_{\lambda,T}(X)=2\min_{\sum_i \phi_i=1}E(\phi_i,T,\lambda, X_i- \bar X_i)\stackrel{\eqref{eq:movebarX}}{=}2\min_{\sum_i \phi_i=1}\lt[E(\phi_i,T,\lambda) +\frac{\phi_i}{T_\lambda} |X_i- \bar X_i|^2\rt],
 \]
which concludes the proof.
\end{proof}

\begin{remark}{\rm 
By the above discussion,  minimizers with either periodic or no boundary condition coincide and thus both the Eulerian formulation and the cone property also hold for minimizers of $E(\Phi,T,\lambda,\bar X)$.}
\end{remark}
By Proposition \ref{prop:reduce} and Lemma \ref{lem:rescalingE}, Theorem \ref{thm:was1} is a direct consequence of the following result.
\begin{proposition}\label{prop:refineddirac}
 For $\Phi, T,\lambda>0$ let $X$ be a minimizer of $E(\Phi,T,\lambda)$. If $t=0$ is a branching point, i.e.  there exists $N\ge 2$ and $\phi_i>0$ with $\sum_{i=1}^N \phi_i=\Phi$ such that there exists a partition of $(-1/2,1/2)=\cup_{i=1}^N A_i$ with $|A_i|=\phi_i$ and
 \[
  X(t,x)\neq X(t,y) \qquad \textrm{for every } t>0 \textrm{ and every } x\in A_i, \ y\in A_j \textrm{ with } i\neq j.
 \]
 then for every $i$,
 \begin{equation}\label{estimbranch}
   \frac{\lambda^2 \Phi^3}{1+\lambda T}\frac{\phi_i}{\Phi}\ges 1.
 \end{equation}
% \[
%   E(\Phi,T,\lambda)=\sum_{i=1}^N E(\phi_i,T,\lambda,-\bar{X}_i),
%  \]
\end{proposition}
\begin{proof}
 By Lemma \ref{lem:rescalingE}, we can assume without loss of generality that $\Phi=1$.  We first prove that
\begin{equation}\label{firstbranchestim}
 \frac{\lambda^2}{1+\lambda T}\ge 1.
\end{equation}
We set
 \[
  A=[-1/2,1/2]\]
  so that by \eqref{eq:cone} we have for $x\in A$,
\[
 X(T,x)\in \lt[-\frac{T}{2T_\lambda},\frac{T}{2 T_\lambda}\rt].
\]
Therefore
\begin{equation}\label{conseqcone}
 \sup_A |X(T,x)|\le \frac{T}{2 T_\lambda}\le \frac{\lambda T}{1+\lambda T}.
\end{equation}

We construct a competitor $\what{X}$ for the energy by setting for $x\in A$,
	\[
	\widehat{X}(t,x)= \begin{cases} 
	0 &\textrm{for } t\in[0,T],\\
		\lambda (t-T) x  &\textrm{for } t\in [T, T_\lambda].
	\end{cases}
	\]
	Notice that $\what{X}(x,0)=X(x,0)=0$, $\what{X}(x,T)=0$ and $\what{X}(x,T_\lambda)=X(x,T_\lambda)=x$. 
%	\begin{figure}\label{fig1}
%		\includegraphics[scale=.7]{DGRfig1.png}
%		\caption{}\label{DGRfig1}
%	\end{figure}
%\ed{Figure 3 is to be erased, right? We never mention it..}

	On the one hand, since $0$ was a branching point at time $t=0$, by the no-loop property, there are  at least two branches for every $t\in(0,T)$. On the other hand, for the competitor we constructed we have a single branch, that is
	\[
	 \int_{A}\int_{0}^T \frac{1}{\phi_X(X(t,x),t)} dt dx\ge 2 T \qquad \textrm{and }
	\qquad \int_{A}\int_{0}^T \frac{1}{\phi_{\what{X}}(\what{X}(t,x),t)}dt dx= T.\]
Moreover, 
\[
 \int_{A} \int_{0}^T |\partial_t \widehat{X}|^2=0 \le \int_{A} \int_{0}^T |\partial_t X|^2.
\]
By minimality of $X$ for $E(T,\lambda)$, 
\begin{equation*}
		\begin{aligned}
			0\le \cL_{\lambda,T,1}(\what{X})-\cL_{\lambda,T,1}(X)&\le -T +\int_{A} \int_{T}^{T_\lambda}|\partial_t \widehat{X}|^2-|\partial_t X|^2 dt dx \\
			&= -T +\lambda \int_{A} |x|^2-|x-X(T,x)|^2\,dx\\
			&= -T + \lambda\int_{A} 2 x X(T,x) -|X(T,x)|^2 dx\\
			&\le -T +\lambda \sup_{A}|X(T,x)|\\
			&\stackrel{\eqref{conseqcone}}{\le}-T + \frac{\lambda^2T}{1+\lambda T}.
			\end{aligned}
	\end{equation*}
	This concludes the proof of \eqref{firstbranchestim}. We now prove that for $i=1,\cdots, N$,
	\begin{equation}\label{secondbranchestim}
	 \frac{\lambda^2}{1+\lambda T}\phi_i\ges 1.
	\end{equation}
For definiteness we prove it for $\phi_1$. We set $\phi=\sum_{i=2}^N \phi_i$. We may assume that $\phi\ge 1/2$ since otherwise $\phi_1\ge 1/2$ and there is nothing to prove. We will make a construction similar to the previous one. It turns out that this construction is simpler to describe via  the Eulerian viewpoint. We set   for $t\in[0,T]$
\[
 \mu^1_t=X(t,\cdot)_\sharp \chi_{A_1} \qquad \textrm{and } \qquad \mu^2_t=X(t,\cdot)_\sharp \chi_{A\backslash A_1}.
\]
We thus have for $j\in\{1,2\}$ and  $t\in [0,T)$, $\mu^j_t=\sum_i \phi_i^j \delta_{X_i^j}$, with (recall that $X(T,\cdot)$ is monotone)
\[
 \cL_{\lambda,T,1}(X)=	I(\mu^1,(0,T))+I(\mu^2,(0,T))+ \lambda W_2^2(\mu^1_T,\chi_{A_1})+ \lambda W_2^2(\mu^2_T,\chi_{A\backslash A_1}).
\]
We then use as competitor $\tilde{\mu}_t=\phi^{-1}\mu^2_t $. Since
\[
 I(\tilde{\mu},(0,T))=P(\mu^2,(0,T))+\phi^{-1} \Ecin(\mu^2,(0,T))
\]
and since $P(\mu^1,(0,T))\ge T$ we find by minimality of $\mu$,
\begin{multline*}
 T+ \Ecin(\mu^2,(0,T))+ \lambda W_2^2(\mu^1_T,\chi_{A_1})+ \lambda W_2^2(\mu^2_T,\chi_{A\backslash A_1})\\
 \le \phi^{-1} \Ecin(\mu^2,(0,T))+\lambda W_2^2((1+\phi^{-1})\mu^2_T,\chi_A).
\end{multline*}
Using that (notice that there is actually equality)
\[
 W_2^2((1+\phi^{-1})\mu^2_T,\chi_A)\le W_2^2(\mu^2_T,\chi_{A\backslash A_1})+W_2^2(\phi^{-1}\mu^2_T,\chi_{A_1})
\]
and $\phi^{-1}-1=(\phi_1/\phi)$, we obtain after simplification that
\begin{equation}\label{keyestimbranch}
 T\le \frac{\phi_1}{\phi} \Ecin(\mu^2,(0,T)) +\lambda\lt[W_2^2(\phi^{-1}\mu^2_T,\chi_{A_1})-W_2^2(\mu^1_T,\chi_{A_1})\rt].
\end{equation}
We now estimate separately the two terms on the right-hand side of \eqref{keyestimbranch}. For the first one we use that by \eqref{conseqcone} we can use Lemma \ref{lem:consblock} with $r=\lambda T/(1+\lambda T)$ to construct a competitor $\what{\mu}$ with $\what{\mu}_0=\mu^2_0$, $\what{\mu}_T= \mu_T^2$ and
\[
 I(\what{\mu},(0,T))\les T +\frac{r^2 \phi}{T}\les T +\frac{\lambda^2 T}{(1+\lambda T)^2}\stackrel{\eqref{firstbranchestim}}{\les} \frac{\lambda^2 T}{1+\lambda T}.
\]
By minimality of $\mu$ we have
\begin{equation}\label{eq:firstestimrhskey}
 \Ecin(\mu^2,(0,T))\le I(\mu^2,(0,T))\le I(\what{\mu},(0,T))\les \frac{\lambda^2 T}{1+\lambda T}.
\end{equation}
For the second term we use triangle and Young inequalities to infer that for every $\delta\in(0,1)$,
\[
 W_2^2(\phi^{-1}\mu^2_T,\chi_{A_1})\le (1+\delta)W_2^2(\mu^1_T,\chi_{A_1}) +\frac{C}{\delta} W_2^2(\phi^{-1}\mu^2_T,\mu^1_T)
\]
and thus
\[
 W_2^2(\phi^{-1}\mu^2_T,\chi_{A_1})-W_2^2(\mu^1_T,\chi_{A_1})\les \delta W_2^2(\mu^1_T,\chi_{A_1}) +\delta^{-1} W_2^2(\phi^{-1}\mu^2_T,\mu^1_T).
\]
Using that
\[
 W_2^2(\mu^1_T,\chi_{A_1})\les \phi_1 \qquad \textrm{and } \qquad
 W_2^2(\phi^{-1}\mu^2_T,\mu^1_T)\stackrel{\eqref{conseqcone}}{\les} \phi_1 \lt(\frac{\lambda T}{1+\lambda T}\rt)^2,
\]
we get after optimization in $\delta$,
\[
 W_2^2(\phi^{-1}\mu^2_T,\chi_{A_1})-W_2^2(\mu^1_T,\chi_{A_1})\les \phi_1\min_{\delta\in(0,1)} \lt[ \delta  +\delta^{-1} \lt(\frac{\lambda T}{1+\lambda T}\rt)^2\rt]\les \phi_1 \frac{\lambda T}{1+\lambda T}.
\]
Plugging this and \eqref{eq:firstestimrhskey} into \eqref{keyestimbranch} yields (recall that $\phi\ge 1/2$)
\[
 T\les \phi_1 \frac{\lambda^2 T}{1+\lambda T}
\]
which after division by $T$ concludes the proof of \eqref{secondbranchestim}.

\end{proof}

We finally prove Theorem \ref{thm:was2}. For this we need the analog of Proposition \ref{prop:scaling} for $\Eund_{\lambda,T}$. Using the Lagrangian formulation we can easily provide a complementary lower bound.

\begin{proposition}\label{propstima}
	For every $\Phi, T,\lambda>0$,
	\begin{equation}\label{estimE}
	 E(\Phi,T,\lambda)\sim T+ \frac{ \lambda \Phi^3}{1+\lambda T}.
	\end{equation}
As a consequence,
	\begin{equation}\label{eqstima}
		\min_{\mu\in \calA_T^*} \Eund_{\lambda,T}(\mu) \sim T+\min\lt(T^{\frac{1}{3}},\lambda^{\frac{1}{3}}T^{\frac{2}{3}}\rt).
	\end{equation}
\end{proposition}

\begin{proof}
We first prove \eqref{estimE}. By \eqref{eq:rescalingE} we may assume without loss of generality that $\Phi=1$. If $\lambda T\le 1$,
we simply take as competitor the non branching measure given by
\[
 X(t,x)=\begin{cases}
         0 & \textrm{if } t\in [0,T]\\
         \lambda (t-T)x & \textrm{if } t\in [T,T_\lambda].
        \end{cases}
\]
If instead $\lambda T\ge 1$ we apply Lemma \ref{lem:consblock} with $r=1$. This concludes the proof of 
\[
 E(T,\lambda)\les  T +\frac{\lambda}{1+\lambda T}.
\]
Regarding the lower bound we have by the Lagrangian formulation \eqref{lagrangianDirac},
\[
 E(T,\lambda)\ge  T+ \int_{-\frac{1}{2}}^{\frac{1}{2}} \frac{|x|^2}{T_\lambda} dx\ges T +\frac{\lambda}{1+\lambda T}.
\]
We now turn to \eqref{eqstima}. By Proposition \ref{prop:reduce},
\begin{multline*}
 \min_{\mu\in \calA_T^*} \Eund_{\lambda,T}(\mu)= 2\min_{\sum_{i=1}^N \phi_i=1} E(\phi_i,T,\lambda)\stackrel{\eqref{estimE}}{\sim} \min_{\sum_{i=1}^N \phi_i=1} \lt[ N T +\frac{\lambda }{1+\lambda T}\sum_{i=1}^N \phi_i^3\rt]\\
 \sim \min_{N\ge 1}\lt[ NT + \frac{\lambda }{1+\lambda T} N^{-2}\rt].
\end{multline*}
Taking
\[
 N=\max\lt(1, \lt(\frac{\lambda}{T(1+\lambda T)}\rt)^{\frac{1}{3}}\rt)
\]
concludes the proof of \eqref{eqstima}.

\end{proof}

We finally prove Theorem \ref{thm:was2} which states that if $\lambda\ll 1$ or $\lambda T\ll1$ then minimizers of $\Eund_{\lambda,T}$ are vertical segments uniformly spaced.

\begin{proof}[Proof of Theorem \ref{thm:was2}]
Let $\mu$ be a minimizer of $\Eund_{\lambda,T}$ and $X$ be the corresponding minimizer of $\cL_{\lambda,T}$. Let us first prove that if $\mu$ does not have branching points then it must be made of uniformly spaced vertical segments. By Proposition \ref{prop:reduce}, we get that it is made of vertical segments. If $\mu_t=\sum_{i=1}^N \phi_i \delta_{X_i}$ (where now $\phi_i$ and $X_i$ are constant in $t$),
\[
 \Eund_{\lambda,T}(\mu)=2 \lt[ NT +\frac{\lambda}{12}\sum_{i=1}^N \phi_i^3\rt].
\]
In order to minimize the right-hand side we must have $\phi_i=1/N$ (which means uniformly spaced segments) and
\[
 N\sim \max\lt(1,\frac{\lambda}{T}\rt)^{\frac{1}{3}}.
\]
We are thus left with the proof that if $\lambda \min(1, T)\ll 1$ then there can be no branching points. The case $\lambda\ll 1$ is directly taken care of by Proposition \ref{prop:refineddirac} since in that case \eqref{estimbranch} cannot hold. We now assume that $\lambda\ge 1$ with $\lambda T\ll 1$. Let $c_0$ be the implicit constant in \eqref{estimbranch} (notice that we actually only use \eqref{firstbranchestim} so we could take $c_0=1$). If $\mu_0=\sum_{i} \phi_i \delta_{X_i}$, we set
\[
 \calI=\{ i \ : \ \phi_i> c_0\lambda^{-2/3}\}.
\]
 By Proposition \ref{prop:reduce},
	\[
	\Eund_{\lambda,T}(\mu)=2\sum_{i} E(\phi_i,T,\lambda)=2\lt(\sum_{\calI^c}E(\phi_i,T,\lambda) +\sum_{\calI} E(\phi_i,T,\lambda)\rt).
	\]
	 We claim that $\calI=\emptyset$. This would conclude the proof since for  $i\in \calI^c$, i.e. $\phi_i\le c_0 \lambda^{-2/3}$, the minimizers of $E(\phi_i,T,\lambda)$ cannot have branching points by \eqref{estimbranch}.
	Let $\Phi= \sum_{\calI} \phi_i$ and assume that $\Phi\neq 0$. On the one hand, by Proposition \ref{prop:reduce},
	\[
	\min_{\sum \psi_i =\Phi} E(\psi_i,T,\lambda)=
    \sum_{\calI} E(\phi_i,T,\lambda)\stackrel{\eqref{estimE}}{\ges} \lambda \sum_{\calI} \phi_i^3\ges \lambda^{-\frac{1}{3}} \sum_{\calI} \phi_i= \lambda^{-\frac{1}{3}} \Phi.
	\]
On the other hand, using \eqref{estimE} with $\psi_i=1/N$, we have for every $N\in \N$,
\[
 \min_{\sum \psi_i =\Phi} E(\psi_i,T,\lambda)\les N T +\frac{\lambda \Phi^3}{N^2}.
\]
Taking $N\sim  \Phi \lambda^{1/3} T^{-1/3}\ges (\lambda T)^{-1/3}\gg1$ we have
\[
 \min_{\sum \psi_i =\Phi} E(\psi_i,T,\lambda)\les  \lambda^{\frac{1}{3}}T^{-\frac{2}{3}} \Phi.
\]
Recalling that we assumed $\lambda T\ll 1$ we have
\[
 \lambda^{\frac{1}{3}}T^{-\frac{2}{3}}\ll \lambda^{-\frac{1}{3}}
\]
from which we reach a contradiction.
\end{proof}    

\begin{remark}{\rm 
		In the regime $\lambda \min(1,T)\ges 1$, we expect that using the scaling properties from Lemma \ref{lem:rescalingE} and the recursive structure given by Proposition \ref{prop:reduce} one can adapt the proof of \cite{G} to exactly identify the minimizers of $E(\lambda,T)$. This goes however beyond the scope of the paper.
		}
\end{remark}

\section*{Acknowledments}
M. Goldman and B. Ruffini are partially supported by the ANR SHAPO. B. Ruffini was partially supported by the INDAM-GNAMPA project CUP-E55F22000270001. G. De Philippis has been partially supported by the NSF grant DMS 2055686 and by the Simons Foundation.

%%%%%%%%%%%%%%%%%%%%%%%%%%%%%%%%%%%%%%%%%%%%%%%%%%%%%%%%
%%%%%%%%%%%%%%%%%%%%%%%%%%%%%%%%%%%%%%%%%%%%%%%%%%%%%%%%
%%%%%%%%%%%%%%%%%%%%%%%%%%%%%%%%%%%%%%%%%%%%%%%%%%%%%%%%	
		%\bibliographystyle{amsplain}

	\bibliographystyle{acm}
\bibliography{OT}

% 	\begin{thebibliography}{99}
% 		
% 		\bibitem{AGS} Ambrosio, Luigi; Gigli, Nicola; Savar\'e, Giuseppe Gradient flows in metric spaces and in the space of probability measures, Second edition. Lectures in Mathematics ETH Z\"urich. Birkh\"auser Verlag, Basel, 2008. x+334 pp.
% 		
% 		\bibitem{AC} Agueh, Martial; Carlier, Guillaume Barycenters in the Wasserstein space. SIAM J. Math. Anal. 43 (2011), no. 2, 904-924.
% 		
% 		\bibitem{CGOS} Conti, Sergio; Goldman, Michael; Otto, Felix; Serfaty, Sylvia A branched transport limit of the Ginzburg-Landau functional. J. \'Ec. polytech. Math. 5 (2018), 317-375.
% 	
% 		\bibitem{privateCOS} Conti, Sergio;  Otto, Felix; Serfaty, Sylvia, Private communication ??
% 	
% 		\bibitem{G} M. Goldman 
% 		Self-similar minimizers of a branched transport functional (2018); (Accepted Paper: Indiana)
% 		
% 		\bibitem{M} Mattila, Pertti Geometry of sets and measures in Euclidean spaces. Fractals and rectifiability. Cambridge Studies in Advanced Mathematics, 44. Cambridge University Press, Cambridge, 1995. xii+343 pp.
% 		
% 		\bibitem{V1} Villani, C\'edric, Topics in optimal transportation. Graduate Studies in Mathematics, 58. American Mathematical Society, Providence, RI, 2003.
% 		
% 		
% 	\end{thebibliography}

\end{document}